\documentclass[10pt]{amsart}
\usepackage{pgf,tikz}

\usepackage{braids}
\usetikzlibrary{decorations.pathreplacing}

\usepackage{caption}

\usepackage{amssymb}
\usepackage{stmaryrd}
\usepackage{mathrsfs}
\usepackage[all]{xy}

\newcommand{\DR}{\mathrm{DR}}
\newcommand{\B}{\mathrm{B}}

\usepackage{rotating}
\usepackage{amscd}
\usepackage{epsfig}

\author{Benjamin Enriquez}
\author{Hidekazu Furusho}

\address{Institut de Recherche Math\'{e}matique Avanc\'{e}e, UMR 7501, 
Universit\'{e} de Strasbourg et CNRS, 7
rue Ren\'{e} Descartes, 67000 Strasbourg, France}
\email{enriquez@math.unistra.fr}

\address{Graduate School of Mathematics, Nagoya University, 
Furo-cho, Chikusa-ku, Nagoya, 464-8602, Japan}
\email{furusho@math.nagoya-u.ac.jp}

\date{February 26, 2022}

\newtheorem{thm}{Theorem}[section]
\newtheorem{lem}[thm]{Lemma}
\newtheorem{lemdef}[thm]{Lemma-Definition}
\newtheorem{cor}[thm]{Corollary}
\newtheorem{prop}[thm]{Proposition}

{\theoremstyle{definition} \newtheorem{rem}[thm]{Remark}}
{\theoremstyle{definition} \newtheorem{defn}[thm]{Definition}}

{\theoremstyle{remark} }

\numberwithin{equation}{subsection}
\numberwithin{figure}{section}

\begin{document}

\baselineskip 16pt 

\title[The Betti side of the double shuffle theory. III. Bitorsor structures]{The Betti side of the 
double shuffle theory. \\   III. Bitorsor structures}

\begin{abstract} 
In the two first parts of the series, we constructed stabilizer subtorsors of a `twisted Magnus' torsor,
studied their relations with the associator and double shuffle torsors, and explained their `de
Rham' nature. In this paper, we make the associated bitorsor structures explicit and explain
the `Betti' nature of the corresponding right torsors; we thereby complete one aim of the
series. We study the discrete and pro-$p$ versions of the `Betti' group of the double shuffle
bitorsor.
\end{abstract}

\bibliographystyle{amsalpha+}
\maketitle

{\footnotesize \tableofcontents}

\section*{Introduction}

This paper is a sequel of \cite{EF1,EF2}. The main result of this series of papers is the  
proof, independent of \cite{Fur}, of the inclusion of the torsor of associators $\mathsf M(\mathbf k)$ 
over a commutative $\mathbb Q$-algebra $\mathbf k$ into the torsor $\mathsf{DMR}^{\DR,\B}(\mathbf k)$ of 
solutions of the double shuffle equations over the same algebra. This is obtained by constructing 
in [EF1] pairs of algebra coproducts $\hat\Delta^{\mathcal W,?}$ and of module coproducts 
$\hat\Delta^{\mathcal M,?}$, $?\in\{\B,\DR\}$, by studying their relation with 
associators, and by studying in \cite{EF2} the relations of $\mathsf M(\mathbf k)$ and 
$\mathsf{DMR}^{\DR,\B}(\mathbf k)$ with the torsors of isomorphisms 
$\mathsf{Stab}(\hat\Delta^{\mathcal W,\DR/\B})(\mathbf k)$ and 
$\mathsf{Stab}(\hat\Delta^{\mathcal M,\DR/\B})(\mathbf k)$ related to these coproducts. 

 It is well-known that the category of torsors (i.e. pairs $(G,X)$, where $G$ is a group 
and $X$ is a set with a free and transitive action of $G$) is equivalent to that of bitorsors
(i.e. triples $(G,X,H)$ where $(G,X)$ is a torsor and $H$ is a group with a right action on $X$, 
free, transitive and commuting with the action of $G$). An example of a bitorsor is the set of 
associators $\mathsf M(\mathbf k)$, where $G$ (resp. $H$) is the Grothendieck-Teichm\"uller group 
$\mathsf{GT}(\mathbf k)$ (resp. $\mathsf{GRT}(\mathbf k)$). Related examples of torsors 
are the `double shuffle' pairs $(\mathsf{DMR}^\DR(\mathbf k),\mathsf{DMR}^{\DR,\B}(\mathbf k))$ and 
$(\mathsf{DMR}_0(\mathbf k),\mathsf{DMR}_\mu(\mathbf k))$, with $\mu\in\mathbf k^\times$ 
(see \cite{Rac,EF2}). 

The primary purpose of this paper is an explicit construction of the corresponding bitorsor 
structures. This is obtained in Theorem \ref{thm:3:13:1705}, (a), together with the analogous explicit 
constructions for the stabilizer torsors $\mathsf{Stab}(\hat\Delta^{\mathcal W,\DR/\B})(\mathbf k)$ 
and $\mathsf{Stab}(\hat\Delta^{\mathcal M,\DR/\B})(\mathbf k)$. In particular, the double shuffle counterpart 
$\mathsf{DMR}^\B(\mathbf k)$ of the group $\mathsf{GT}(\mathbf k)$ is obtained in terms of a 
stabilizer of the coproduct $\hat\Delta^{\mathcal M,\B}$ (Lemma-Definition \ref{lem:3:9:1705}). 

The other results of the paper are concerned with the study of the group $\mathsf{DMR}^\B(\mathbf k)$. 
We study its relation with the group $\mathsf{GT}(\mathbf k)$ (Theorem \ref{thm:3:13:1705}, (b)). 
We give an equivalent definition of $\mathsf{DMR}^\B(\mathbf k)$ in terms of group-like elements for the 
coproduct $\hat\Delta^{\mathcal M,\B}$  (\S\ref{sect:4:0307}). Analogously 
to what is done in \cite{Dr} for the group scheme $\mathbf k\mapsto\mathsf{GT}(\mathbf k)$, 
we construct and study discrete and pro-$p$ versions of the group scheme 
$\mathbf k\mapsto\mathsf{DMR}^\B(\mathbf k)$ (\S\S\ref{section:discrete:group}, \ref{section:pro-l}). 

The identification of the bitorsor structure of the torsor $\mathsf{DMR}^{\DR,\B}(\mathbf k)$, obtained 
in Theorem \ref{thm:3:13:1705}, (a), is based on its relation with a stabilizer torsor (see \cite{EF2}, \S3.1) 
and with the identification of the bitorsor of a stabilizer torsor (Lemma \ref{lem:1:6}). The computation 
of the discrete analogue $\mathrm{DMR}^\B$ of $\mathsf{DMR}^\B(\mathbf k)$ (\S\ref{sect:5:2:3006}) is based 
on the study of the discrete version $\Delta^{\mathcal M,\B}$ of the coproduct $\hat\Delta^{\mathcal M,\B}$. 
In order to establish the main properties of its pro-$p$ version $\mathrm{DMR}^\B_p$ (\S\ref{section:final}), 
we recall some basics on pro-$p$ groups (\S\ref{section:plapucodg}), and prove statements of \cite{Dr} on 
the pro-$p$ analogue $\mathrm{GT}_p$ of $\mathsf{GT}(\mathbf k)$ (\S\ref{subsection:rugtl}).

The material is distributed as follows: \S\ref{sect:tab:0707} is devoted to basic material on torsors and bitorsors, 
\S\ref{sect:tbgaia:0707} is devoted to the construction of an explicit bitorsor $\mathsf G^{\DR,\B}(\mathbf k)$ related to the 
twisted Magnus group, \S\ref{sog:0707} is devoted to the construction of several subbitorsors of 
$\mathsf G^{\DR,\B}(\mathbf k)$ and to the first main result (Theorem \ref{thm:3:13:1705}), namely the construction of a
bitorsor structure on $\mathsf{DMR}^{\DR,\B}(\mathbf k)$ and of the group $\mathsf{DMR}^\B(\mathbf k)$, 
\S\ref{sect:4:0307} is devoted to  an equivalent definition of this group, \S\ref{section:discrete:group} is devoted 
to the definition and computation of a discrete analogue of $\mathsf{DMR}^\B(\mathbf k)$, and \S\ref{section:pro-l} 
is devoted to the construction of its pro-$p$ analogue. 

\subsubsection*{Notation} In all the paper, $\mathbf k$ is a commutative and associative $\mathbb Q$-algebra 
with unit. For $A$ an algebra, we denote by $A^\times$ the group of its invertible elements. For $a\in A$ (resp. 
$u\in A^\times$), we denote by $\ell_a$ (resp. $\mathrm{Ad}_u$) the self-map of $A$ given by $x\mapsto ax$ 
(resp. $x\mapsto uxu^{-1}$). 

\subsubsection*{Acknowledgements} The collaboration of both authors has been supported by grant JSPS KAKENHI
JP15KK0159 and JP18H01110. 

\section{Torsors and bitorsors}\label{sect:tab:0707}

In \S\ref{subsect:bitorsors:0506}, we recall the formalism of torsors. We introduce the category 
of bitorsors in \S\ref{sect:1:2:1305} and recall its equivalence with the category of torsors. 
In \S\ref{sect:1:3:0207}, we recall the main torsors of \cite{EF2} and their interrelations. 

\subsection{Bitorsors}\label{subsect:bitorsors:0506}

Recall from \cite{EF2} the definitions of a (left) torsor, of a morphism between two torsors, of a subtorsor of a torsor 
(Definition 2.1), of the preimage of a subtorsor by a torsor morphism (Lemma 2.2), of the intersection of two 
subtorsors of a torsor (Lemma 2.3), of the trivial torsor $_GG$ attached to a group $G$ (Lemma 2.4), of the 
torsor injection $\mathrm{inj}_a:\ _HH\to\ _GG$  attached to a group inclusion $H\subset G$
and an element $a\in G$ (Lemma 2.5), of a stabilizer subtorsor (Lemma 2.6). 

\begin{defn} (see \cite{Gi}, Chap. III, Definition 1.5.3)
(a) A {\it bitorsor} $_GX_H$ is a triple $(G,X,H)$, where $X$ is a set and $G$ and $H$ are groups, equipped with commuting
 left and right actions of $G$ and $H$ on $X$, which are both free and transitive.  

(b) If $_GX_H$ and $_{G'}X'_{H'}$ are bitorsors, a {\it bitorsor morphism} $_GX_H\to\ _{G'}X'_{H'}$ is the data of a map 
$X\to X'$ and of compatible group morphisms $G\to G'$ and $H\to H'$. 

(c) A {\it subbitorsor} of the bitorsor $_{G'}X'_{H'}$ is the data of a subset $X$ of $X'$ and of subgroups $G,H$ of  
$G',H'$, such that the inclusions build up a torsor morphism $_GX_H\to\ _{G'}X'_{H'}$ , which is then a {\it bitorsor inclusion}. 
\end{defn}

We will use the expressions `the torsor $X$' (resp. `the bitorsor $X$') to designate a torsor $_GX$ (resp. a bitorsor $_GX_H$). 

\begin{lem}\label{lem:1:2}
Let $_GX_H\to\ _{G'}X'_{H'}\leftarrow _{G''}X''_{H''}$ be a diagram of a bitorsors. 
Let $G'''$ (resp. $X'''$, $H'''$) be the fibered product of $G$ (resp. $X$, $H$) 
and $G''$ (resp. $X''$, $H''$) above $G'$ (resp. $X'$, $H'$). Then either $X'''$ is empty, or 
$_{G'''}X'''_{H'''}$ is a bitorsor, called the fibered product of 
$_GX_H$ and $_{G''}X''_{H''}$ above $_{G'}X'_{H'}$, denoted $X'\times_{X}X''$ and fitting in a commutative diagram of bitorsors 
$$
\xymatrix{ _{G'''}X'''_{H'''}\ar[r]\ar[d] & _{G''}X''_{H''}\ar[d]  \\ _{G'}X'_{H'}\ar[r] & _{G}X_{H}}
$$
If $_{G''}X''_{H''}$ is a subbitorsor of 
$_{G'}X'_{H'}$ and $X'''$ is nonempty, then $_{G'''}X'''_{H'''}$ is a subbitorsor of
of $_GX_H$, called the preimage of $_{G''}X''_{H''}$ by the torsor morphism $_GX_H\to\ _{G'}X'_{H'}$. 
\end{lem}

\proof Obvious. \hfill\qed\medskip 

\begin{defn}\label{def:Cartesian}
We call a commutative square $$\xymatrix{A\ar[r]\ar[d]&B\ar[d]\\C\ar[r]&D}$$ of bitorsors Cartesian iff there is an isomorphism to bitorsors $A\to B\times_DC$
such that the diagram 
$$\xymatrix{A\ar[r]\ar[d]\ar[rd]&B\\C&\ar[l]\ar[u]B\times_D C}$$
commutes. 
\end{defn}

\begin{lem}\label{lem:1:3}
Let $_GX_H$ be a bitorsor and let $_{G'}X'_{H'}$ and $_{G''}X''_{H''}$ be subbitorsors. Then either $X'\cap X''$ is empty, or
$_{G'\cap G''}X'\cap X''_{H'\cap H''}$ is a subbitorsor of $_GX_H$, called the intersection of both bitorsors. 
\end{lem}

\proof Obvious. \hfill\qed\medskip 

\begin{lem}\label{lem:1:4}\label{construction(*)}
A group isomorphism $i:G'\to G$ gives rise to a bitorsor $_GG_{G'}$, where $G$ acts on the left on itself, and the right action 
of $G'$ on $G$ is given by $g\cdot g':=g i(g')$. The bitorsor $_GG_G$ corresponding to $i$ being the identity map of $G$
is called the trivial bitorsor attached to $G$. 
\end{lem}

\proof Obvious. \hfill\qed\medskip 

\begin{lem}\label{lem:1:5}
Let $G$ be a group and $H$ be a subgroup. For any $a\in G$, there is a torsor inclusion 
$\mathrm{\mathrm{inj}}_a:\ _HH_H\to\ _GG_G$ 
where the first (resp. second, third) component is the inclusion $H\hookrightarrow G$ (resp. the map $h\mapsto ha^{-1}$, 
the group morphism $H\to G$, $h\mapsto aha^{-1}$). 
\end{lem}

\proof Obvious. \hfill\qed\medskip 

The map $a\mapsto\mathrm{\mathrm{inj}}_a(_GG_G)$ sets up a map $G/H\to\{$subbitorsors of 
$_GG_G\}$. 

If $\mathcal C$ is a category, we denote by $\mathrm{Iso}_{\mathcal C}(O,O')$ the set of isomorphisms between
two objects $O$ and $O'$ and by $\mathrm{Aut}_{\mathcal C}(O)$ the group of automorphisms of an object $O$.

\begin{defn}
The bitorsor attached to a pair of isomorphic objects $O,O'$ of a category $\mathcal C$ is 
$\mathrm{Bitor}(O,O'):=\ _{\mathrm{Aut}_{\mathcal C}(O)}\mathrm{Iso}_{\mathcal C}(O',O)_{
\mathrm{Aut}_{\mathcal C}(O')}$. An action of a bitorsor on the pair 
$(O,O')$ is a morphism from this bitorsor to $\mathrm{Bitor}(O,O')$. 
\end{defn}

\begin{lem}\label{construction(**)}
Let $\mathcal C$ be a category and let $O,O'$ be objects. Let $i:G'\to G$ be a group isomorphism and let 
$G\to\mathrm{Aut}_{\mathcal C}(O)$, $G'\to\mathrm{Aut}_{\mathcal C}(O')$ be group morphisms, 
denoted $g\mapsto g_O$ and $g'\mapsto g'_{O'}$. Let $i_{O',O}\in\mathrm{Iso}_{\mathcal C}(O',O)$
be such that for any $g'\in G'$, one has $i(g')_O=i_{O'O}\circ g'_{O'}\circ (i_{O'O})^{-1}$. Then 
a bitorsor morphism $_GG_{G'}\to\mathrm{Bitor}(O,O')$ is given by the above group morphisms and the map 
$G\to\mathrm{Iso}_{\mathcal C}(O',O)$, $g\mapsto g_O\circ i_{O'O}$. 
\end{lem}

\proof Obvious. \hfill\qed\medskip 

\begin{lem}\label{*:2804}
Let $_GX_H$ be a bitorsor acting on a pair 
$(O,O')$ of isomorphic objects of a $\mathbf k$-linear tensor category
$\mathcal C$. Then for any $n,m\geq0$, 
$(\mathrm{Hom}_{\mathcal C}(O^{\otimes n},O^{\otimes m}), 
\mathrm{Hom}_{\mathcal C}(O^{\prime\otimes n},O^{\prime\otimes m}))$
is a pair of isomorphic $\mathbf k$-modules, equipped with an action of  
$_GX_H$. 
\end{lem}

\proof If the action on $(O,O')$ is denoted $G\ni g\mapsto g_O\in\mathrm{Aut}_{\mathcal C}(O)$, 
$X\ni x\mapsto x_{O',O}\in\mathrm{Iso}_{\mathcal C}(O',O)$, 
$H\ni h\mapsto h_{O'}\in\mathrm{Aut}_{\mathcal C}(O')$, then the 
action on this pair of modules is as follows: $g\in G$ acts by 
$\mathrm{Hom}_{\mathcal C}(O^{\otimes n},O^{\otimes m})\ni 
f\mapsto g_O^{\otimes m}\circ f\circ(g_O^{\otimes n})^{-1}$, 
$h\in H$ acts by 
$\mathrm{Hom}_{\mathcal C}(O^{\prime\otimes n},O^{\prime\otimes m})
\ni f'\mapsto h_{O'}^{\otimes m}\circ f'\circ(h_{O'}^{\otimes n})^{-1}$, 
$x\in X$ acts by 
$\mathrm{Hom}_{\mathcal C}(O^{\prime\otimes n},O^{\prime\otimes m})
\ni f'\mapsto x_{O',O}^{\otimes m}\circ f'\circ (x_{O',O}^{\otimes n})^{-1}
\in \mathrm{Hom}_{\mathcal C}(O^{\otimes n},O^{\otimes m})$.  
\hfill\qed\medskip  

\begin{lem}\label{lem:1:6}
Let $_GX_H$ be a bitorsor acting on a pair $(V,W)$ of isomorphic $\mathbf k$-modules. If $(v,w)\in V\times W$, 
then either $\mathrm{Stab}_X(v,w):=\{x\in X|x\cdot w=v\}$ is empty, or 
$_{\mathrm{Stab}_G(v)}\mathrm{Stab}_X(v,w)_{\mathrm{Stab}_{H}(w)}$ is a subbitorsor of 
$_GX_H$, called the stabilizer bitorsor of $(v,w)$, where $\mathrm{Stab}_G(v):=\{g\in G|g\cdot v=v\}$ and 
$\mathrm{Stab}_{H}(w):=\{h\in H|h\cdot w=w\}$. 
\end{lem}

\proof Similar to that of \cite{EF2}, Lemma 2.6. \hfill\qed\medskip 

\begin{lem} \label{constr:bitors:tannakia}
Let $\mathcal C$ be a $\mathbb Q$-linear neutral Tannakian category and let 
$\omega_1,\omega_2:\mathcal C\to\mathrm{Vec}_{\mathbb Q}$ be two fiber functors. 
The set $\mathrm{Iso}^{\otimes}(\omega_2,\omega_1)$ is a bitorsor under the 
left and right actions of $\mathrm{Aut}^{\otimes}(\omega_1)$ and $\mathrm{Aut}^{\otimes}(\omega_2)$. 
\end{lem}

\proof Obvious. \hfill\qed\medskip 

Define a {\it right torsor} $X_H$ to be a pair $(X,H)$ of a set $X$ and a group $H$ acting freely and transitively on $X$ from 
the right. This gives rise to a torsor $_{H^{\mathrm{op}}}X$. 
Let $\mathbf{Tor}$ and $\mathbf{Bitor}$ be the categories of torsors and bitorsors. Then there are two functors 
$\mathbf{Bitor}\to\mathbf{Tor}$, defined on objects by $_GX_H\mapsto\ _GX$ and $_GX_H\mapsto\ 
_{H^{\mathrm{op}}}X$.

\subsection{From torsors to bitorsors}\label{sect:1:2:1305}

For $_G\!X$ a torsor, define $\mathrm{Aut}_G(X)$ to be subgroup of the permutation group of $X$ acting on 
the right which commute with all the elements of $G$. Then $\mathrm{Aut}_G(X)$ is a group, equipped with a right action on 
$X$ which commutes with the left action of $G$. 

\begin{lem}\label{lem:gattt}
(a) The assignment $_G\!X\mapsto\mathrm{Aut}_G(X)$ defines a functor $\mathbf{Tor}\to\mathbf{Gp}$ from the 
category of torsors to that of groups. The group $\mathrm{Aut}_G(X)$ will be called the `group attached to the torsor $_G\!X$'. 

(b) If $_G\!X$ is a torsor, then the right action of $\mathrm{Aut}_G(X)$ defines a bitorsor structure $_G\!X_{\mathrm{Aut}_G(X)}$ 
on it.  The assignment $_G\!X\mapsto\ _G\!X_{\mathrm{Aut}_G(X)}$ defines a functor $\mathbf{Tor}\to\mathbf{Bitor}$, 
quasi-inverse to the functor $\mathbf{Bitor}\to\mathbf{Tor}$, $_GX_H\mapsto\ _G\!X$. In particular, if $_GX_H$ is a bitorsor, 
then there is a canonical isomorphism $H\simeq\mathrm{Aut}_G(X)$. 
\end{lem}

\proof See \cite{Gi}, Chap. III, Proposition 1.5.1. \hfill\qed\medskip 

\begin{lem}\label{lem:1:10:1305}
If $_GX_H$ is a bitorsor and if $_{G'}X'_{H'}$, $_{G''}X''_{H''}$ are subbitorsors of $_GX_H$ such that 
$_{G'}X'$ is a subtorsor of $_{G''}X''$, then $H'\subset H''$ (inclusion of subgroups of $H$). 
\end{lem}

\proof Any element $x\in X$ gives rise to a group isomorphism $\mathrm{Ad}_{x^{-1}}: G\to H$, defined by the identity 
$g\cdot x=x\cdot\mathrm{Ad}_{x^{-1}}(g)$ for any $g\in G$, $x\in X$. If $_{G'}X'_{H'}$ is a subbitorsor of $_GX_H$ and if 
$x\in X'$, then $\mathrm{Ad}_{x^{-1}}$ restricts to a group isomorphism $G'\to H'$. Likewise, $\mathrm{Ad}_{x^{-1}}$ 
restricts to a group isomorphism $G''\to H''$. Then the subgroup $H'$ (resp. $H''$) of $H$ is the  image of the subgroup $G'$ 
(resp. $G''$) of $G$ by $\mathrm{Ad}_{x^{-1}}$. As $G'\subset G''$, this implies $H'\subset H''$.  \hfill\qed\medskip

\subsection{Some torsors and their interrelations}\label{sect:1:3:0207}

In \cite{EF2}, we introduced a `twisted Magnus' torsor 
$
\mathsf G^{\DR,\B}(\mathbf k)$ (\S2.2), together 
with its subtorsors $
\mathsf G^{\DR,\B}_{\mathrm{quad}}(\mathbf k)$ (\S2.3), 
$
\mathsf{Stab}
(\hat\Delta^{\mathcal X,\DR/\B})(\mathbf k)$ ($\mathcal X\in\{\mathcal W,\mathcal M\}$) 
(\S\S2.6, 2.7), 
$
\mathsf M(\mathbf k)$ (\S2.4 and \cite{Dr}), and 
$
\mathsf{DMR}^{\DR,\B}(\mathbf k)$ (\S2.5), which is 
constructed using the torsor 
$
\mathsf{DMR}_\mu(\mathbf k)$ for $\mu\in\mathbf k^\times$ 
(\cite{Rac}) and contains it as a subtorsor. 

In \S3.1, Theorem 3.1, we showed the following relations between these subtorsors: 

(1) inclusion $\mathsf M(\mathbf k)\hookrightarrow\mathsf{DMR}^{\DR,\B}(\mathbf k)$; 

(2) equality $\mathsf{DMR}^{\DR,\B}(\mathbf k)=\mathsf{Stab}
(\hat\Delta^{\mathcal M,\DR/\B})(\mathbf k)\cap 
\mathsf G^{\DR,\B}_{\mathrm{quad}}(\mathbf k)$; 

(3) inclusion $\mathsf{Stab}
(\hat\Delta^{\mathcal M,\DR/\B})(\mathbf k)\hookrightarrow 
\mathsf{Stab}(\hat\Delta^{\mathcal W,\DR/\B})(\mathbf k)$.

The main purpose of this paper is the explicit computation of the 
groups attached to these torsors in the sense of Lemma \ref{lem:gattt}, (a), as well as the study of their interrelations. 
This will make explicit the bitorsor structures attached to these torsors by Lemma \ref{lem:gattt}, (b). 

\section{The bitorsor $\mathsf G^{\DR,\B}(\mathbf k)$ and its actions}\label{sect:tbgaia:0707}

In \S\ref{sect:2:1:1505}, we construct a Betti counterpart $(\mathsf G^\B(\mathbf k),\circledast)$ of the 
variant $\mathsf G^\DR(\mathbf k)$ from \cite{EF2} of the `twisted Magnus group' from \cite{Rac}. 
In \S\ref{sect:2:2:02:07}, we construct actions of $\mathsf G^\B(\mathbf k)$, which are 
the Betti counterparts of the actions of $\mathsf G^\DR(\mathbf k)$ from \cite{EF2}, \S1.6. In 
\S\ref{sect:2:3:02:07}, we show how to use $\mathsf G^\B(\mathbf k)$ for upgrading the torsor
structure  
$
\mathsf G^{\DR,\B}(\mathbf k)$ into a bitorsor one, and 
in \S\ref{sect:2:4:2804}, we show how the actions of \S\ref{sect:2:2:02:07} and of \cite{EF2}, 
\S1.6 can be combined to construct an action of the bitorsor $\mathsf G^{\DR,\B}(\mathbf k)$. 

\subsection{The group $(\mathsf G^\B(\mathbf k),\circledast)$}\label{sect:2:1:1505}

\subsubsection{The group $(\mathcal G(\hat{\mathcal V}^\B),\cdot)$}\label{sect:2:1:1:2606}

Let $\mathcal V^\B$  be the $\mathbf k$-algebra introduced in \cite{EF1}, \S2.1, and let 
$(F^i\mathcal V^\B)_{i\geq 0}$ be the decreasing algebra filtration on it defined in \cite{EF1}, \S2.1. 
Let $\Delta^{\mathcal V,\B}:\mathcal V^\B\to(\mathcal V^\B)^{\otimes2}$ be the morphism of 
filtered $\mathbf k$-algebras introduced in \cite{EF1}, \S2.3. Then $(\mathcal V^\B,\Delta^{\mathcal V,\B})$
is a Hopf algebra. In \cite{EF1}, \S2.5, we introduced the topological Hopf algebra 
$(\hat{\mathcal V}^\B,\hat\Delta^{\mathcal V,\B})$ obtained by completion with respect to the filtrations. 

Denote by $\mathcal G({\mathcal V}^\B)$ (resp. $\mathcal G(\hat{\mathcal V}^\B)$) be the 
group of group-like elements of the  Hopf algebra $(\mathcal V^\B,\Delta^{\mathcal V,\B})$ (resp. 
$(\hat{\mathcal V}^\B,\hat\Delta^{\mathcal V,\B})$); we denote by $\cdot$ its product. 
According to \cite{EF1}, \S2.1, the Hopf algebra 
$(\mathcal V^\B,\Delta^{\mathcal V,\B})$ is the group algebra $\mathbf kF_2$, where $F_2$ is the 
free group with generators $X_0,X_1$, which leads to the equality $\mathcal G(\mathcal V^\B)=F_2$.
Then the natural morphism $\mathcal G({\mathcal V}^\B)\to\mathcal G(\hat{\mathcal V}^\B)$ can be identified
with the morphism from $F_2$ to its $\mathbf k$-prounipotent completion. 

\subsubsection{The automorphisms $\mathrm{aut}^{\mathcal V,(1),\B}_{(\mu,g)}$ and 
$\mathrm{aut}^{\mathcal V,(10),\B}_{(\mu,g)}$.}\label{sect:212}

Let $\hat{\mathcal V}^\B_1:=1+F^1\hat{\mathcal V}^\B\subset\hat{\mathcal V}^\B$. 
For $a\in\hat{\mathcal V}^\B_1$, $\mathrm{log}(a)$ may be expanded as a series in $a-1$  
and therefore makes sense as an element of $F^1\hat{\mathcal V}^\B$. If now $\mu\in\mathbf k$, 
$\mathrm{exp}(\mu\mathrm{log}(a))$ may be expanded as a series in $\mu\mathrm{log}(a)$ and 
therefore makes sense as an element of $\hat{\mathcal V}^\B_1$. This defines a self-map 
of $\hat{\mathcal V}^\B_1$, denoted $a\mapsto a^\mu$, which restricts to a self-map of 
$\mathcal G(\hat{\mathcal V}^\B)$. 

For $(\mu,g)\in\mathbf k^\times\times\mathcal G(\hat{\mathcal V}^\B)$, one checks that there is a unique automorphism 
$\mathrm{aut}^{\mathcal V,(1),\B}_{(\mu,g)}$ of the topological $\mathbf k$-algebra $\hat{\mathcal V}^\B$, 
such that 
\begin{equation}\label{aut:V:1:1705}
\mathrm{aut}^{\mathcal V,(1),\B}_{(\mu,g)} : X_0\mapsto g\cdot X_0^\mu\cdot g^{-1}, \quad X_1
\mapsto X_1^\mu
\end{equation}
and a unique automorphism $\mathrm{aut}^{\mathcal V,(10),\B}_{(\mu,g)}$ of the topological $\mathbf k$-module 
$\hat{\mathcal V}^\B$, 
such that 
\begin{equation}\label{aut:V:10:1705}
\forall a\in \hat{\mathcal V}^\B, \quad 
\mathrm{aut}^{\mathcal V,(10),\B}_{(\mu,g)}(a)=\mathrm{aut}^{\mathcal V,(1),\B}_{(\mu,g)}(a)\cdot g. 
\end{equation}

\subsubsection{The group $(\mathsf G^\B(\mathbf k),\circledast)$}\label{sect:constr:pdt:1704}

\begin{lem}\label{lem:2:1:2606}
The product $\circledast$ given by 
$$
(\mu,g)\circledast(\mu',g'):=(\mu\mu',\mathrm{aut}^{\mathcal V,(10),\B}_{(\mu,g)}(g')) 
$$
defines a group structure on $\mathsf G^\B(\mathbf k):=\mathbf k^\times\times\mathcal G(\hat{\mathcal V}^\B)$. 
\end{lem}

\proof When $(\mu,g)\in\mathsf G^\B(\mathbf k)$, $\mathrm{aut}^{\mathcal V,(1),\B}_{(\mu,g)}$ is an automorphism of the 
Hopf algebra $(\hat{\mathcal V}^\B,\hat\Delta^{\mathcal V,\B})$, which implies that for any 
$g'\in\mathcal G(\hat{\mathcal V}^\B)$, $\mathrm{aut}^{\mathcal V,(1),\B}_{(\mu,g)}(g')\in
\mathcal G(\hat{\mathcal V}^\B)$. It follows that $\mathrm{aut}^{\mathcal V,(10),\B}_{(\mu,g)}(g')\in
\mathcal G(\hat{\mathcal V}^\B)$, and therefore that $\circledast$ is a well-defined map $\mathsf G^\B(\mathbf k)^2\to
\mathsf G^\B(\mathbf k)$, which restricts to a map $\mathcal G(\hat{\mathcal V}^\B)^2\to\mathcal G(\hat{\mathcal V}^\B)$. 
The fact that $(\mathcal G(\hat{\mathcal V}^\B),\circledast)$ is a group can be proved analogously to \cite{Rac}, 
Proposition 3.1.6. The group $\mathbf k^\times$ acts on $\hat{\mathcal V}^\B$ by $\mu\bullet X_i:=X_i^\mu$ for $i=0,1$. 
This induces an action of $\mathbf k^\times$ on $(\mathcal G(\hat{\mathcal V}^\B),\circledast)$; one checks that 
$(\mathsf G^\B(\mathbf k),\circledast)$ is the corresponding semidirect product, the injections $\mathbf k^\times\hookrightarrow
\mathsf G^\B(\mathbf k)$ and $\mathcal G(\hat{\mathcal V}^\B)\hookrightarrow\mathsf G^\B(\mathbf k)$ being given by 
$\mu\mapsto(\mu,1)$ and $g\mapsto (1,g)$. \hfill\qed\medskip 

\subsection{Actions of the group $(\mathsf G^\B(\mathbf k),\circledast)$}\label{sect:2:2:02:07}

 We denote by $\mathbf k$-alg (resp. $\mathbf k$-mod) the category of  $\mathbf k$-algebras (resp.
 $\mathbf k$-modules).

\begin{defn} \label{def:aut:A:M} (see \cite{EF2}, Definition 1.1)
If $A$ is a $\mathbf k$-algebra and $M$ is a left $A$-module, then $\mathrm{Aut}(A,M)$ is the set of pairs $(\alpha,\theta)
\in\mathrm{Aut}_{\mathbf k\operatorname{-alg}}(A)\times\mathrm{Aut}_{\mathbf k\operatorname{-mod}}(M)$, such that for any 
$a\in A$, $m\in M$, one has $\theta(am)=\alpha(a)\theta(m)$; this is a subgroup of $\mathrm{Aut}_{\mathbf k\operatorname{-alg}}(A)
\times\mathrm{Aut}_{\mathbf k\operatorname{-mod}}(M)$. 
\end{defn}

\begin{lem}
For $(\mu,g)\in\mathsf G^\B(\mathbf k)$, the pair $(\mathrm{aut}_{(\mu,g)}^{\mathcal V,(1),\B},
\mathrm{aut}_{(\mu,g)}^{\mathcal V,(10),\B})$ belongs to $\mathrm{Aut}(\hat{\mathcal V}^\B,\hat{\mathcal V}^\B)$
(in which $\hat{\mathcal V}^\B$ is viewed as the left regular module over itself).  
The map $(\mathsf G^\B(\mathbf k),\circledast)\to\mathrm{Aut}(\hat{\mathcal V}^\B,\hat{\mathcal V}^\B)$, 
$(\mu,g)\mapsto (\mathrm{aut}_{(\mu,g)}^{\mathcal V,(1),\B},\mathrm{aut}_{(\mu,g)}^{\mathcal V,(10),\B})$ 
is a group morphism. 
\end{lem}

\proof The proof of the first statement is similar to that of the corresponding statement in Lemma 1.10 of \cite{EF2}.  
The fact that $(\mathsf G^\B(\mathbf k),\circledast)\to\mathrm{Aut}_{\mathbf k\operatorname{-alg}}(\hat{\mathcal V}^\B)$, 
$(\mu,g)\mapsto \mathrm{aut}_{(\mu,g)}^{\mathcal V,(1),\B}$ (resp. $(\mathsf G^\B(\mathbf k),\circledast)\to
\mathrm{Aut}_{\mathbf k\operatorname{-mod}}(\hat{\mathcal V}^\B)$, $(\mu,g)\mapsto 
\mathrm{aut}_{(\mu,g)}^{\mathcal V,(10),\B}$) is a group morphism  is proved similarly to 
Lemma 1.8 (resp. Lemma 1.9) in \cite{EF2}. These facts imply the second statement. 
\hfill\qed\medskip 

\begin{lem} \label{lemma:2:4:12042020}(see \cite{EF2}, Lemma 1.2)
In the situation of Definition \ref{def:aut:A:M}, let $A_0$ be a subalgebra of $A_0$ and $M_0$ be an $A$-submodule of $M$. 
The set $\mathrm{Aut}(A,M|A_0,M_0)$ of pairs $(\alpha,\theta)\in\mathrm{Aut}(A,M)$  such that $\alpha$ (resp. 
$\theta$) restricts to an automorphism of $A_0$ (resp. of $M_0$) is a subgroup of $\mathrm{Aut}(A,M)$, moreover
there is a natural group morphism $\mathrm{Aut}(A,M|A_0,M_0)\to\mathrm{Aut}(A_0,M/M_0)$. 
\end{lem}

Let $\hat{\mathcal W}^\B$ be the topological $\mathbf k$-subalgebra 
of $\hat{\mathcal V}^\B$) defined in \cite{EF1}, \S2.5; it is given by 
$\hat{\mathcal W}^\B=\mathbf k1\oplus\hat{\mathcal V}^\B(X_1-1)$.  
View $\hat{\mathcal V}^\B$ as the left regular topological module over 
itself; then $\hat{\mathcal V}^\B(X_0-1)$ is a topological submodule. 
According to \cite{EF1}, \S2.5, we denote the corresponding quotient 
topological $\hat{\mathcal V}^\B$-module by $\hat{\mathcal M}^\B$ and 
by $1_\B$ the class of $1$ in this module; according to \cite{EF1}, 
\S2.5, the restriction of $\hat{\mathcal M}^\B$ to 
$\hat{\mathcal W}^\B$ is free of rank one, generated by $1_\B$. 

\begin{lemdef}\label{lemma:2:5:12042020}
If $(\mu,g)\in\mathsf G^\B(\mathbf k)$, then $(\mathrm{aut}_{(\mu,g)}^{\mathcal V,(1),\B},
\mathrm{aut}_{(\mu,g)}^{\mathcal V,(10),\B})$ is in 
$$
\mathrm{Aut}(\hat{\mathcal V}^\B,\hat{\mathcal V}^\B|\hat{\mathcal W}^\B,\hat{\mathcal V}^\B(X_0-1)).
$$  
We denote by $(\mathrm{aut}_{(\mu,g)}^{\mathcal W,(1),\B},\mathrm{aut}_{(\mu,g)}^{\mathcal M,(10),\B})$ the corresponding 
element of $\mathrm{Aut}(\hat{\mathcal W}^\B,\hat{\mathcal M}^\B)$ (see Lemma \ref{lemma:2:4:12042020}). The map 
taking $(\mu,g)$ to this element  is a group morphism $(\mathsf G^\B(\mathbf k),\circledast)\to
\mathrm{Aut}(\hat{\mathcal W}^\B,\hat{\mathcal M}^\B)$. 
\end{lemdef}

\proof It follows from $\hat{\mathcal W}^\B=\mathbf k1\oplus\hat{\mathcal V}^\B(X_1-1)$, from 
$\mathrm{aut}_{(\mu,g)}^{\mathcal V,(1),\B}(X_1)=X_1^\mu$ for any $(\mu,g)\in\mathsf G^\B(\mathbf k)$, and from $X_1^\mu-1=f_\mu(X_1-1)\cdot 
(X_1-1)$, where $f_\mu(t):=((1+t)^\mu-1)/t\in\mathbf k[[t]]$ for any $\mu\in\mathbf k^\times$, that 
$\mathrm{aut}_{(\mu,g)}^{\mathcal V,(1),\B}$ restricts to an automorphism of $\hat{\mathcal W}^\B$ for 
any $(\mu,g)\in\mathsf G^\B(\mathbf k)$.  

If $(\mu,g)\in\mathsf G^\B(\mathbf k)$ and $v\in\hat{\mathcal V}^\B$, then $\mathrm{aut}_{(\mu,g)}^{\mathcal V,(10),\B}
(v\cdot (X_0-1))=\mathrm{aut}_{(\mu,g)}^{\mathcal V,(1),\B}(v\cdot (X_0-1))\cdot g=
\mathrm{aut}_{(\mu,g)}^{\mathcal V,(1),\B}(v)\cdot g\cdot (X_0^\mu-1)=
\mathrm{aut}_{(\mu,g)}^{\mathcal V,(10),\B}(v) f_\mu(X_0-1)\cdot (X_0-1)$; this implies that 
$\mathrm{aut}_{(\mu,g)}^{\mathcal V,(10),\B}$
restricts to an automorphism of $\hat{\mathcal V}^\B(X_0-1)$. 
\hfill\qed\medskip 

Recall that $\hat{\mathcal V}^\B$ is freely generated, as a topological algebra, by $\mathrm{log}X_0$ and $\mathrm{log}X_1$; 
for $g\in\hat{\mathcal V}^\B$, we denote by $\{\mathrm{log}X_0,\mathrm{log}X_1\}^*\to\mathbf k$ the map such that 
$g=\sum_{w\in\{\mathrm{log}X_0,\mathrm{log}X_1\}^*}(g|w)w$, where $\{\mathrm{log}X_0,\mathrm{log}X_1\}^*$
is the set of (possibly empty) words in $\mathrm{log}X_0,\mathrm{log}X_1$. 
  
\begin{defn}\label{def:Gamma:Betti} 
For $g\in\hat{\mathcal V}^\B$, $\Gamma_g\in\mathbf k[[t]]^\times$ is the series given by 
$$
\Gamma_g(t):=\mathrm{exp}(\sum_{n\geq 1}(-1)^{n+1}(g|(\mathrm{log}X_0)^{n-1}\mathrm{log}X_1)t^n/n).
$$ 
\end{defn}

\begin{defn} (see \cite{EF2}, Definition 1.1)
If $A$ is a $\mathbf k$-algebra and if $G\to\mathrm{Aut}_{\mathbf k\operatorname{-alg}}(A)$, $g\mapsto\alpha_g$ is a group morphism, 
a cocycle of $G$ in $A$ equipped with $g\mapsto\alpha_g$ is a map $G\to A^\times$, 
$g\mapsto c_g$ such that $c_{gg'}=c_g\cdot\alpha_g(c_{g'})$ for any $g,g'\in G$.   
\end{defn}

\begin{lem}\label{lemma:cocycle}
The map $\Gamma:\mathsf G^\B(\mathbf k)\to(\hat{\mathcal V}^\B)^\times$, $(\mu,g)\mapsto
\Gamma_g^{-1}(-\mathrm{log}X_1)$, is a cocycle of $(\mathsf G^\B(\mathbf k),\circledast)$
in $\hat{\mathcal V}^\B$ equipped with $(\mu,g)\mapsto\mathrm{aut}_{(\mu,g)}^{\mathcal V,(1),\B}$. 
This map corestricts to a map $\Gamma:\mathsf G^\B(\mathbf k)\to(\hat{\mathcal W}^\B)^\times$, which is 
a cocycle of the same group in $\hat{\mathcal W}^\B$ equipped with 
$(\mu,g)\mapsto\mathrm{aut}_{(\mu,g)}^{\mathcal W,(1),\B}$. 
\end{lem}

\proof The proof is similar to that of Lemmas 1.12 and 1.13 in \cite{EF2}. \hfill\qed\medskip 

\begin{lemdef}\label{lem:2:9:0105}
For $(\mu,g)\in\mathsf G^\B(\mathbf k)$, set ${}^\Gamma\!\mathrm{aut}_{(\mu,g)}^{\mathcal W,(1),\B}:=
\mathrm{Ad}_{\Gamma_g^{-1}(-\mathrm{log}X_1)}\circ
\mathrm{aut}_{(\mu,g)}^{\mathcal W,(1),\B}\in\mathrm{Aut}_{\mathbf k\operatorname{-}\mathrm{alg}}(\hat{\mathcal W}^\B)$, 
${}^\Gamma\!\mathrm{aut}_{(\mu,g)}^{\mathcal M,(10),\B}:=\ell_{\Gamma_g^{-1}(-\mathrm{log}X_1)}\circ
\mathrm{aut}_{(\mu,g)}^{\mathcal M,(10),\B}\in\mathrm{Aut}_{\mathbf k\operatorname{-}\mathrm{mod}}(\hat{\mathcal M}^\B)$.
Then the pair $({}^\Gamma\!\mathrm{aut}_{(\mu,g)}^{\mathcal W,(1),\B},
{}^\Gamma\!\mathrm{aut}_{(\mu,g)}^{\mathcal M,(10),\B})$ belongs to 
$\mathrm{Aut}(\hat{\mathcal W}^\B,\hat{\mathcal M}^\B)$ and the map $(\mathsf G^\B(\mathbf k),\circledast)
\to\mathrm{Aut}(\hat{\mathcal W}^\B,\hat{\mathcal M}^\B)$ is a group morphism. 
\end{lemdef}

\proof  Lemma 1.4 in \cite{EF2} gives conditions for a cocycle to be used for twisting a group 
morphism $G\to\mathrm{Aut}(A,M)$.  Lemmas  \ref{lemma:cocycle} and \ref{lemma:2:5:12042020}  
show that these conditions are fulfilled. \hfill\qed\medskip 

\subsection{A bitorsor structure on $\mathsf G^{\DR,\B}(\mathbf k)$}\label{sect:2:3:02:07}

The group $(\mathsf G^\DR(\mathbf k),\circledast)$ is defined in \cite{EF2}, \S1.6.1. Recall that 
$\mathsf G^\DR(\mathbf k)=\mathbf k^\times\times\mathcal G(\hat{\mathcal V}^\DR)$ (equality of sets), 
where $(\hat{\mathcal V}^\DR,\hat\Delta^{\mathcal V,\DR})$ is the topological Hopf algebras 
from \cite{EF1}, \S1.3, and the product $\circledast$ is defined in \cite{EF2}, \S1.6.1.
In \cite{EF1}, \S3.3, we defined an isomorphism $\mathrm{iso}^{\mathcal V}:
\hat{\mathcal V}^\B\to\hat{\mathcal V}^\DR$ of topological $\mathbf k$-algebras.  

\begin{lem}\label{lem:iso:1704}
$\mathrm{iso}^{\mathcal V}$ is an isomorphism of topological Hopf algebras. 
\end{lem} 

\proof This follows from $\mathrm{iso}^{\mathcal V}:X_i\mapsto\mathrm{exp}(e_i)$ for $i=0,1$, 
and from the fact the $X_i$ (resp. $e_i$) is group-like (resp. primitive) for $\hat\Delta^{\mathcal V,\B}$
(resp. $\hat\Delta^{\mathcal V,\DR}$). \hfill\qed\medskip 

\begin{lem}\label{lem:gp:iso:1704}
The map $\mathrm{iso}^{\mathcal V}$ restricts to group isomorphisms 
$\mathrm{iso}^{\mathcal G}:(\mathcal G(\hat{\mathcal V}^\B),\circledast)\to
(\mathcal G(\hat{\mathcal V}^\DR),\circledast)$ and $\mathrm{iso}^{\mathsf G}:
(\mathsf G^\B(\mathbf k),\circledast)\to(\mathsf G^\DR(\mathbf k),\circledast)$, given by the map 
$\mathsf G^\B(\mathbf k)=\mathbf k^\times\times\mathcal G(\hat{\mathcal V}^\DR)
\stackrel{\mathrm{id}\times \mathrm{iso}^{\mathcal G}}\to
\mathbf k^\times\times\mathcal G(\hat{\mathcal V}^\DR)=\mathsf G^\DR(\mathbf k)$. 
\end{lem}

\proof The fact that $\mathrm{iso}^{\mathcal G}$ and $\mathrm{iso}^{\mathsf G}$ are bijections follows from Lemma 
\ref{lem:iso:1704}. The compatibility of these maps with the products $\circledast$ on both sides follows from comparison of 
their constructions, given respectively in \S\ref{sect:constr:pdt:1704} and \cite{EF2}, \S1.6.1. \hfill\qed\medskip 

\begin{defn}\label{def:2:12:1244}
The bitorsor $\mathsf G^{\DR,\B}(\mathbf k)$ is the structure obtained by applying Lemma \ref{construction(*)} to the 
group isomorphism from Lemma \ref{lem:gp:iso:1704}. 
\end{defn}

The underlying set of this bitorsor is $\mathsf G^\DR(\mathbf k)$, the left action is that of $\mathsf G^\DR(\mathbf k)$ on itself, 
and the right action is that of $\mathsf G^\B(\mathbf k)$ on $\mathsf G^\DR(\mathbf k)$ given by $(\mu,g)\cdot(\mu',g'):=
(\mu,g)\circledast \mathrm{iso}^{\mathsf G}(\mu',g')$. The corresponding torsor coincides therefore with the torsor 
$\mathsf G^{\DR,\B}(\mathbf k)$ from \cite{EF2}, Definition 2.8. 

\subsection{Actions of the bitorsor $\mathsf G^{\DR,\B}(\mathbf k)$}\label{sect:2:4:2804}

\begin{defn} We define $\mathbf{k}\operatorname{-alg-mod}$ to be the category whose 
objects are the pairs $(A,M)$, where $A$ is a $\mathbf k$-algebra and 
$M$ is $A$-module, and such that a morphism $(A,M)\to (A',M')$ is a pair of 
an algebra morphism $A\to A'$ and a compatible $\mathbf k$-module 
morphism $M\to M'$. 
\end{defn}

In \cite{EF1}, \S1.3, we defined the topological $\mathbf k$-subalgebra of $\hat{\mathcal V}^\DR$ by 
$\hat{\mathcal W}^\DR:=\mathbf k1\oplus\hat{\mathcal V}^\DR e_1$, the $\hat{\mathcal V}^\DR$-module 
$\hat{\mathcal M}^\DR:=\hat{\mathcal V}^\DR/\hat{\mathcal V}^\DR e_0$, and its element $1_\DR$ defined 
as the class of $1$, and showed that $\hat{\mathcal M}^\DR$ 
is freely generated by $1_\DR$ as a $\hat{\mathcal W}^\DR$-module.  Recall from \cite{EF1}, \S3.3
that the isomorphism $\mathrm{iso}^{\mathcal V}$ induces compatible isomorphisms 
$\mathrm{iso}^{\mathcal W}:\hat{\mathcal W}^\B\to\hat{\mathcal W}^\DR$ 
of topological $\mathbf k$-algebras
and $\mathrm{iso}^{\mathcal M}:\hat{\mathcal M}^\B\to\hat{\mathcal M}^\DR$ of
topological $\mathbf k$-modules. 

\begin{lem} 
For $\omega\in\{\B,\DR\}$, the pair $(\hat{\mathcal W}^\omega,\hat{\mathcal M}^\omega)$, in which the second component is viewed as a left
module over the first, is an object of $\mathbf{k}\operatorname{-}\mathrm{alg}\operatorname{-}\mathrm{mod}$, which will be denoted 
$(\hat{\mathcal W},\hat{\mathcal M})^\omega$. An isomorphism $\mathrm{iso}^{\mathcal W,\mathcal M}:
(\hat{\mathcal W},\hat{\mathcal M})^\B\to(\hat{\mathcal W},\hat{\mathcal M})^\DR$ is given by the pair $(\mathrm{iso}^{\mathcal W},\mathrm{iso}^{\mathcal M})$. 
\end{lem}

\proof The first statement follows from \cite{EF1}, \S\S1.3 and 2.5. The second statement follows from 
\cite{EF1}, \S3.3. \hfill\qed\medskip 

In \cite{EF2}, Lemma 1.14, we defined a group morphism
$$
(\mathsf G^\DR(\mathbf k),\circledast)\to\mathrm{Aut}_{\mathbf k\operatorname{-alg-mod}}((\hat{\mathcal W},\hat{\mathcal M})^\DR),\quad  
(\mu,g)\mapsto(^\Gamma\!\mathrm{aut}^{\mathcal W,(1),\DR}_{(\mu,g)}, 
\ ^\Gamma\!\mathrm{aut}^{\mathcal M,(10),\DR}_{(\mu,g)}).
$$ 

\begin{lem}\label{lem:2104}
The conditions of Lemma \ref{construction(**)} are satisfied in case of the group isomorphism 
$\mathrm{iso}^{\mathsf G}:\mathsf G^\B(\mathbf k)\to\mathsf G^\DR(\mathbf k)$, of the isomorphism 
$\mathrm{iso}^{\mathcal W,\mathcal M}:
(\hat{\mathcal W},\hat{\mathcal M})^\B\to
(\hat{\mathcal W},\hat{\mathcal M})^\DR$ 
in $\mathbf{k}\operatorname{-}\mathrm{alg}\operatorname{-}\mathrm{mod}$, and of the group morphisms $\mathsf G^\omega(\mathbf k)\ni(\mu,g)\mapsto (^\Gamma\!\mathrm{aut}^{\mathcal W,(1),\omega}_{(\mu,g)},\ ^\Gamma\!\mathrm{aut}^{\mathcal M,(10),\omega}_{(\mu,g)})
\in\mathrm{Aut}_{\mathbf{k}\operatorname{-}\mathrm{alg}\operatorname{-}\mathrm{mod}}((\hat{\mathcal W},\hat{\mathcal M})^\omega)$ for $\omega\in\{\B,\DR\}$.
\end{lem}

\proof This follows from the identities $\mathrm{aut}^{\mathcal V,(\alpha),\DR}_{\mathrm{iso}^{\mathsf G}(\mu,g)}
=\mathrm{iso}^{\mathcal V}\circ\mathrm{aut}^{\mathcal V,(\alpha),\B}_{(\mu,g)}\circ(\mathrm{iso}^{\mathcal V})^{-1}$
for any $(\mu,g)\in\mathsf G^\B(\mathbf k)$, where $\alpha\in\{1,10\}$, which follow from comparison of the definitions of 
$\mathrm{aut}^{\mathcal V,(\alpha),\B}_{(\mu,g)}$ and $\mathrm{aut}^{\mathcal V,(\alpha),\DR}_{(\mu,g)}$ in \S\ref{sect:212} 
and in \cite{EF2}, \S1.6.2, and from $\Gamma_{\mathrm{iso}^{\mathcal G}(g)}(t)=\Gamma_{g}(t)$
for $g\in\mathsf G^\B(\mathbf k)$. \hfill\qed\medskip 

In \cite{EF2}, Definition 1.15, we defined, for any $(\mu,g)\in\mathsf G^\DR(\mathbf k)$, isomorphisms 
$^\Gamma\!\mathrm{comp}_{(\mu,g)}^{\mathcal W,(1)}=\ ^\Gamma\!\mathrm{aut}_{(\mu,g)}^{\mathcal W,(1),\DR}\circ 
\mathrm{iso}^{\mathcal W}\in\mathrm{Iso}_{\mathbf k\operatorname{-alg}}(\hat{\mathcal W}^\B,\hat{\mathcal W}^\DR)$ and 
$^\Gamma\!\mathrm{comp}_{(\mu,g)}^{\mathcal M,(10)}= ^\Gamma\!\mathrm{aut}_{(\mu,g)}^{\mathcal M,(10),\DR}\circ 
\mathrm{iso}^{\mathcal M}\in\mathrm{Iso}_{\mathbf k\operatorname{-mod}}(\hat{\mathcal M}^\B,\hat{\mathcal M}^\DR)$. 

\begin{lem}
The map $\mathsf G^\DR(\mathbf k)\to\mathrm{Iso}_{\mathbf{k}\operatorname{-}\mathrm{alg}\operatorname{-}\mathrm{mod}}
((\hat{\mathcal W},\hat{\mathcal M})^\B,(\hat{\mathcal W},\hat{\mathcal M})^\DR)$ underlying the bitorsor morphism 
$\mathsf G^{\DR,\B}(\mathbf k)\to\mathrm{Bitor}((\hat{\mathcal W},\hat{\mathcal M})^\DR,(\hat{\mathcal W},\hat{\mathcal M})^\B)$ arising 
from Lemma \ref{lem:2104} is 
$$
(\mu,g)\mapsto(^\Gamma\!\mathrm{comp}_{(\mu,g)}^{\mathcal W,(1)},
\ ^\Gamma\!\mathrm{comp}_{(\mu,g)}^{\mathcal M,(10)}).
$$ 
\end{lem}

\proof Immediate. \hfill\qed\medskip 

\section{Subbitorsors of $\mathsf G^{\DR,\B}(\mathbf k)$}\label{sog:0707}

In this section, we explicitly construct the bitorsors attached to the torsors of \cite{EF2}, 
by constructing Betti counterparts of their underlying groups;  
this is done in \S\ref{sect:3:1} (resp. \S\ref{sect:3:2}, \S\ref{sect:3:3}, \S\ref{sect:3:4}, \S\ref{sect:3:5})
for the 'linear and quadratic conditions' torsor $\mathsf G^{\DR,\B}_{\mathrm{quad}}(\mathbf k)$ 
(resp. the stabilizer torsors $\mathsf{Stab}(\hat{\Delta}^{\mathcal W,\DR/\B})(\mathbf k)$, 
$\mathsf{Stab}(\hat{\Delta}^{\mathcal M,\DR/\B})(\mathbf k)$, the double shuffle torsors 
$\mathsf{DMR}^{\DR,\B}(\mathbf k)$ and $\mathsf{DMR}_\mu(\mathbf k)$), and this construction is 
recalled from \cite{Dr} in \S\ref{sect:3:6} for the case of the torsor of associators $\mathsf M(\mathbf k)$. 
In \S\ref{sect:3:7}, we derive from the relations of the torsors of \cite{EF2} certain
relations between the Betti counterparts of the underlying groups. We make explicit the
Lie algebras of these groups and their interrelations in \S\ref{sect:3:8}. Finally, we make explicit the 
relations of the stabilizer bitorsors $\mathsf{Stab}(\hat{\Delta}^{\mathcal W,\DR/\B})(\mathbf k)$, 
$\mathsf{Stab}(\hat{\Delta}^{\mathcal M,\DR/\B})(\mathbf k)$ with coalgebra and Hopf algebra bitorsors (\S\ref{rwhacaib}).

\subsection{$\mathsf G^{\DR,\B}_{\mathrm{quad}}(\mathbf k)$ as a subbitorsor of 
$\mathsf G^{\DR,\B}(\mathbf k)$}\label{sect:3:1}

\subsubsection{The group $(\mathsf G^\B_{\mathrm{quad}}(\mathbf k),\circledast)$}

\begin{defn}\label{def:G:B:quad:1321}
We set 
$$
\mathsf G^\B_{\mathrm{quad}}(\mathbf k):=\{(\mu,g)\in G^\B(\mathbf k)\ |\ 
(g|\mathrm{log}X_0)=(g|\mathrm{log}X_1)=0, \quad 
\mu^2=1+24(g|\mathrm{log}X_0\mathrm{log}X_1)\}. 
$$
\end{defn}

\begin{lem}\label{lem:3:2:2204}
$\mathsf G^\B_{\mathrm{quad}}(\mathbf k)$ is a subgroup of 
$(\mathsf G^\B(\mathbf k),\circledast)$. 
\end{lem}

\proof  Define $\mathsf G^\B_{\mathrm{lin}}(\mathbf k)$ to be the subset of all pairs $(\mu,g)$ of $\mathsf G^\B(\mathbf k)$ 
such that $(g|\mathrm{log}X_0)= (g|\mathrm{log}X_0)=0$. Then $\mathsf G^\B_{\mathrm{lin}}(\mathbf k)=
(\mathrm{iso}^{\mathsf G})^{-1}(\mathsf G^\DR_{\mathrm{lin}}(\mathbf k))$, where $\mathsf G^\DR_{\mathrm{lin}}(\mathbf k)$ 
is the subgroup of $(\mathsf G^\DR(\mathbf k),\circledast)$ defined in \cite{EF2}, proof of Lemma 2.10, therefore 
$(\mathsf G^\B_{\mathrm{lin}}(\mathbf k),\circledast)$ is a subgroup of $(\mathsf G^\B(\mathbf k),\circledast)$. 

Recall from \cite{EF2}, proof of Lemma 2.10, the group $(\mathbf k^\times\times\mathbf k,\cdot)$, where 
$(\mu,c)\cdot(\mu',c'):=(\mu\mu',c+\mu^2c')$; it is the semidirect product corresponding to the action of 
$\mathbf k^\times$ on $(\mathbf k,+)$ by $\lambda\bullet a:=\lambda^2a$. The composed morphism 
$\varphi_{\mathrm{quad}}\circ\mathrm{iso}^{\mathsf G} : \mathsf G^\B_{\mathrm{lin}}(\mathbf k)\to
(\mathbf k^\times\times\mathbf k,\cdot)$
is given by $(\mu,g)\mapsto (\mu,(g|\mathrm{log}X_0\mathrm{log}X_1))$, where 
$\varphi_{\mathrm{quad}} : \mathsf G^\DR_{\mathrm{lin}}(\mathbf k)\to(\mathbf k^\times\times\mathbf k,\cdot)$ is as in 
\cite{EF2}, proof of Lemma 2.10. 

On the other hand, $\{(\mu,c)|c=(\mu^2-1)/24,\mu\in\mathbf k^\times\}$ coincides with the subgroup 
$\mathrm{Ad}_{(1,-1/24)}(\mathbf k^\times)$ of $(\mathbf k^\times\times\mathbf k,\cdot)$, where $\mathbf k^\times$
is the subgroup $\{(\mu,0)|\mu\in\mathbf k^\times\}$ of $(\mathbf k^\times\times\mathbf k,\cdot)$. Its preimage 
by the above group morphism is therefore a subgroup of $(\mathsf G^\B_{\mathrm{lin}}(\mathbf k),\circledast)$. 
The result now follows from the fact that this preimage coincides with $\mathsf G^\B_{\mathrm{quad}}(\mathbf k)$.
\hfill\qed\medskip 

\subsubsection{Bitorsor structure on $\mathsf G^{\DR,\B}_{\mathrm{quad}}(\mathbf k)$} 

In \cite{EF2}, Definition 2.9, we introduced the subsets $\mathsf G^{\DR,\B}_{\mathrm{quad}}(\mathbf k)$ and 
$\mathsf G^\DR_{\mathrm{quad}}(\mathbf k)$ of $\mathsf G^\DR(\mathbf k)$, and in Lemma 2.10, proved that  
$\mathsf G^\DR_{\mathrm{quad}}(\mathbf k)$ is a subgroup of $\mathsf G^\DR(\mathbf k)$, and that 
$\mathsf G^{\DR,\B}_{\mathrm{quad}}(\mathbf k)$, equipped with the left action of this group, is a 
subtorsor of $\mathsf G^{\DR,\B}(\mathbf k)$. 

\begin{lem}\label{lemma:bitorsor:1}
The right action of $\mathsf G^\B(\mathbf k)$ on $\mathsf G^\DR(\mathbf k)$ restricts to a free and transitive action of 
$\mathsf G^\B_{\mathrm{quad}}(\mathbf k)$ on $\mathsf G^{\DR,\B}_{\mathrm{quad}}(\mathbf k)$, 
so that this right action equips $\mathsf G^{\DR,\B}_{\mathrm{quad}}(\mathbf k)$ with the structure of a 
subbitorsor of the bitorsor $\mathsf G^{\DR,\B}(\mathbf k)$.  
\end{lem}

\proof Let $i:G'\to G$ be a group isomorphism and let $H\subset G$ be a subgroup, 
and $H':=i^{-1}(H)\subset G'$. The $i$ restricts to an isomorphism 
$H'\to H$. This gives rise to bitorsors $_HH_{H'}$
and $_GG_{G'}$ (see Lemma \ref{construction(*)}); the natural inclusions give rise to a bitorsor 
inclusion $_HH_{H'}\subset\ _GG_{G'}$. A group morphism $\varphi:H\to K$
gives rise to group morphism $\varphi\circ i:H'\to K$ and to a bitorsor morphism 
$_{\varphi}\varphi_{\varphi\circ i} :\ _HH_{H'}\to\ _KK_K$. Then a subgroup 
$L\subset K$ and an element $k\in K$ give rise to a bitorsor inclusion 
$\mathrm{inj}_k:\ _LL_L\hookrightarrow\ _KK_K$ (see Lemma \ref{lem:1:5}). Then 
$(_{\varphi}\varphi_{\varphi\circ i})^{-1}(\mathrm{inj}_k(_LL_L))$ is a subtorsor of $_HH_{H'}$, hence of 
$_GG_{G'}$. One has $(_{\varphi}\varphi_{\varphi\circ i})^{-1}(\mathrm{inj}_k(_LL_L))=\ _AB_C$, where 
$A:=\varphi^{-1}(L)\subset G$, $C:=i^{-1}\circ\varphi^{-1}(\mathrm{Ad}_k(L))
\subset G'$, $B:=\varphi^{-1}(L\cdot k^{-1})\subset G$. 
This is summarized in the diagram of bitorsors 
$$
\xymatrix{
& (_{\varphi}\varphi_{\varphi\circ i})^{-1}(\mathrm{inj}_k(_LL_L))\ar@{^{(}->}[r]\ar[d]& 
_HH_{H'}\ar@{^{(}->}[r]\ar^{_{\varphi}\varphi_{\varphi\circ i}}[d]& _GG_{G'}
\\ 
_LL_L \ar@{->>}_{\mathrm{inj}_k}[r]& \mathrm{inj}_k(_LL_L)\ar@{^{(}->}[r]& _KK_K& 
}
$$
When $i:=\mathrm{iso}^{\mathsf G}:\mathsf G^\B(\mathbf k)\to\mathsf G^\DR(\mathbf k)$, 
$H:=\mathsf G^\DR_{\mathrm{lin}}(\mathbf k)$, $H':=\mathsf G^\B_{\mathrm{lin}}(\mathbf k)$,
$K:=(\mathbf k^\times\times\mathbf k,\cdot)$, $\varphi
:=\varphi_{\mathrm{quad}}$ (see \cite{EF2}, proof of Lemma 2.3), 
$L:=\mathbf k^\times=\{(\mu,0)|\mu\in\mathbf k^\times\}$, 
$k:=(1,-1/24)$, one has $A=\varphi_{\mathrm{quad}}^{-1}(\mathbf k^\times)
=\mathsf G_{\mathrm{quad}}^\DR(\mathbf k)$ by \cite{EF2}, Lemma 2.3, 
$C=(\mathrm{iso}^{\mathcal V})^{-1}\circ\varphi_{\mathrm{quad}}^{-1}
(\mathrm{Ad}_{(1,-1/24)}(\mathbf k^\times))=\mathsf G_{\mathrm{quad}}^\B(\mathbf k)$
by Lemma \ref{lem:3:2:2204}. One checks that $\mathbf k^\times\cdot k^{-1}=\{(\mu,c)|
c=\mu^2/24\}$, which shows that $B=\varphi_{\mathrm{quad}}^{-1}(\mathbf k^\times\cdot k^{-1})
=\mathsf G^{\DR,\B}(\mathbf k)$. \hfill\qed\medskip 

\subsection{$\mathsf{Stab}(\hat\Delta^{\mathcal W,\DR/\B})(\mathbf k)$ as a subbitorsor of 
$\mathsf G^{\DR,\B}(\mathbf k)$}\label{sect:3:2}

\begin{lem}
The map $\mathsf G^\B(\mathbf k)\to\mathrm{Aut}_{\mathbf k\operatorname{-mod}}(\mathrm{Hom}_{\mathbf k\operatorname{-mod}_{\mathrm{top}}}
(\hat{\mathcal W}^\B,(\mathcal W^\B)^{\otimes2\wedge}))$ taking  $(\mu,g)$ to $F\mapsto
(^\Gamma\!\mathrm{aut}_{(\mu,g)}^{\mathcal W,(1),\B})^{\otimes2}
\circ F \circ (^\Gamma\!\mathrm{aut}_{(\mu,g)}^{\mathcal W,(1),\B})^{-1}$
defines an action of the group $\mathsf G^\B(\mathbf k)$ on the 
$\mathbf k$-module $\mathrm{Hom}_{\mathbf k\operatorname{-mod}_{\mathrm{top}}}
(\hat{\mathcal W}^\B,(\mathcal W^\B)^{\otimes2\wedge})$. 
\end{lem}

\proof Follows from Lemma-Definition \ref{lem:2:9:0105}. \hfill\qed\medskip 

\begin{defn}
Set $\mathsf{Stab}(\hat\Delta^{\mathcal W,\B})(\mathbf k):=\mathrm{Stab}_{\mathsf G^\B(\mathbf k)}
(\hat\Delta^{\mathcal W,\B})$. This is a subgroup of $(\mathsf G^\B(\mathbf k),\circledast)$, equal to the set of $(\mu,g)\in\mathsf G^B(\mathbf k)$, such that 
$(^\Gamma\!\mathrm{aut}_{(\mu,g)}^{\mathcal W,(1),\B})^{\otimes2}
\circ \hat\Delta^{\mathcal W,\B}=\hat\Delta^{\mathcal W,\B}\circ 
\ ^\Gamma\!\mathrm{aut}_{(\mu,g)}^{\mathcal W,(1),\B}$. 
\end{defn}

In \cite{EF2}, Definitions 2.14 and 2.15, we introduced the subsets $\mathsf{Stab}(\hat\Delta^{\mathcal W,\DR/\B})(\mathbf k)$ and 
$\mathsf{Stab}(\hat\Delta^{\mathcal W,\DR})(\mathbf k)$ of $\mathsf G^\DR(\mathbf k)$, and in Lemma 2.16 and Theorem 3.1, 
proved that $\mathsf{Stab}(\hat\Delta^{\mathcal W,\DR})(\mathbf k)$ is a subgroup of $\mathsf G^\DR(\mathbf k)$, and that 
$\mathsf{Stab}(\hat\Delta^{\mathcal W,\DR/\B})(\mathbf k)$, equipped with the left action of this group, is a subtorsor of 
$\mathsf G^{\DR,\B}(\mathbf k)$. 

\begin{lem}\label{subbitorsor:0105}
The right action of $\mathsf G^\B(\mathbf k)$ on $\mathsf G^\DR(\mathbf k)$ restricts to a free and transitive action of 
$\mathsf{Stab}(\hat\Delta^{\mathcal W,\B})(\mathbf k)$ on $\mathsf{Stab}(\hat\Delta^{\mathcal W,\DR/\B})(\mathbf k)$, 
so that this right action equips $\mathsf{Stab}(\hat\Delta^{\mathcal W,\DR/\B})(\mathbf k)$ with the structure of a 
subbitorsor of the bitorsor $\mathsf G^{\DR,\B}(\mathbf k)$.  
\end{lem}

\proof One derives from \S\ref{sect:2:4:2804} an action of the bitorsor 
$\mathsf G^{\DR,\B}(\mathbf k)$
on the pair of isomorphic topological $\mathbf k$-algebras $(\hat{\mathcal W}^\DR,\hat{\mathcal W}^\B)$. 
Viewing this as a pair of objects of the tensor category $\mathbf k$-mod$_{\mathrm{top}}$, 
and using Lemma \ref{*:2804}, one obtains an action of the bitorsor $\mathsf G^{\DR,\B}(\mathbf k)$
on the pair of isomorphic $\mathbf k$-modules $(\mathrm{Hom}_{\mathbf k\operatorname{-mod}_{\mathrm{top}}}
(\hat{\mathcal W}^\DR,(\mathcal W^\DR)^{\otimes2\wedge}), \mathrm{Hom}_{\mathbf k\operatorname{-mod}_{\mathrm{top}}}
(\hat{\mathcal W}^\B,(\mathcal W^\B)^{\otimes2\wedge}))$. Recall that 
$(\hat{\Delta}^{\mathcal W,\DR},\hat{\Delta}^{\mathcal W,\B})$ is a pair of elements of this pair of modules. 
By Lemma \ref{lem:1:6}, this gives rise to the stabilizer subbitorsor 
$_{\mathrm{Stab}_{\mathsf G^\DR(\mathbf k)}(\hat{\Delta}^{\mathcal W,\DR})}\mathrm{Stab}_{\mathsf G^\DR(\mathbf k)}(\hat{\Delta}^{\mathcal W,\DR},\hat{\Delta}^{\mathcal W,\B})_{\mathrm{Stab}_{\mathsf G^\B(\mathbf k)}(\hat{\Delta}^{\mathcal W,\B})}$ of $\mathsf G^{\DR,\B}(\mathbf k)$. 
\hfill\qed\medskip 

\subsection{$\mathsf{Stab}(\hat\Delta^{\mathcal M,\DR/\B})(\mathbf k)$ as a subbitorsor of 
$\mathsf G^{\DR,\B}(\mathbf k)$}\label{sect:3:3}

\begin{lem}
The map $\mathsf G^\B(\mathbf k)\to\mathrm{Aut}_{\mathbf k\operatorname{-mod}}(\mathrm{Hom}_{\mathbf k\operatorname{-mod}_{\mathrm{top}}}
(\hat{\mathcal M}^\B,(\mathcal M^\B)^{\otimes2\wedge}))$ taking  $(\mu,g)$ to $F\mapsto
(^\Gamma\!\mathrm{aut}_{(\mu,g)}^{\mathcal M,(10),\B})^{\otimes2}
\circ F \circ (^\Gamma\!\mathrm{aut}_{(\mu,g)}^{\mathcal M,(10),\B})^{-1}$
defines an action of the group $\mathsf G^\B(\mathbf k)$ on the 
$\mathbf k$-module $\mathrm{Hom}_{\mathbf k\operatorname{-mod}_{\mathrm{top}}}
(\hat{\mathcal W}^\B,(\mathcal W^\B)^{\otimes2\wedge})$. 
\end{lem}

\proof Follows from Lemma-Definition \ref{lem:2:9:0105}. \hfill\qed\medskip 

\begin{defn}
Set $\mathsf{Stab}(\hat\Delta^{\mathcal M,\B})(\mathbf k):=\mathrm{Stab}_{\mathsf G^\B(\mathbf k)}
(\hat\Delta^{\mathcal M,\B})$. This is a subgroup of $(\mathsf G^\B(\mathbf k),\circledast)$, equal to the set of $(\mu,g)\in\mathsf G^B(\mathbf k)$, such that 
$(^\Gamma\!\mathrm{aut}_{(\mu,g)}^{\mathcal M,(10),\B})^{\otimes2}
\circ \hat\Delta^{\mathcal M,\B}=\hat\Delta^{\mathcal M,\B}\circ 
\ ^\Gamma\!\mathrm{aut}_{(\mu,g)}^{\mathcal M,(10),\B}$. 
\end{defn}

In \cite{EF2}, Definitions 2.17 and 2.18, we introduced the subsets $\mathsf{Stab}(\hat\Delta^{\mathcal M,\DR/\B})(\mathbf k)$ and 
$\mathsf{Stab}(\hat\Delta^{\mathcal M,\DR})(\mathbf k)$ of $\mathsf G^\DR(\mathbf k)$, and in Lemma 2.19 and Theorem 3.1, 
proved that $\mathsf{Stab}(\hat\Delta^{\mathcal M,\DR})(\mathbf k)$ is a subgroup of $\mathsf G^\DR(\mathbf k)$, and that 
$\mathsf{Stab}(\hat\Delta^{\mathcal M,\DR/\B})(\mathbf k)$, equipped with the left action of this group, is a subtorsor of 
$\mathsf G^{\DR,\B}(\mathbf k)$. 

\begin{lem}\label{subbitorsor:0105:bis}
The right action of $\mathsf G^\B(\mathbf k)$ on $\mathsf G^\DR(\mathbf k)$ restricts to a free and transitive action of 
$\mathsf{Stab}(\hat\Delta^{\mathcal M,\B})(\mathbf k)$ on $\mathsf{Stab}(\hat\Delta^{\mathcal M,\DR/\B})(\mathbf k)$, 
so that this right action equips $\mathsf{Stab}(\hat\Delta^{\mathcal W,\DR/\B})(\mathbf k)$ with the structure of a 
subbitorsor of the bitorsor $\mathsf G^{\DR,\B}(\mathbf k)$.  
\end{lem}

\proof One derives from \S\ref{sect:2:4:2804} an action of the bitorsor 
$\mathsf G^{\DR,\B}(\mathbf k)$
on the pair of isomorphic topological $\mathbf k$-algebras $(\hat{\mathcal W}^\DR,\hat{\mathcal W}^\B)$. 
Viewing this as a pair of objects of the tensor category $\mathbf k$-mod$_{\mathrm{top}}$, 
and using Lemma \ref{*:2804}, one obtains an action of the bitorsor $\mathsf G^{\DR,\B}(\mathbf k)$
on the pair of isomorphic $\mathbf k$-modules $(\mathrm{Hom}_{\mathbf k\operatorname{-mod}_{\mathrm{top}}}
(\hat{\mathcal W}^\DR,(\mathcal W^\DR)^{\otimes2\wedge}), \mathrm{Hom}_{\mathbf k\operatorname{-mod}_{\mathrm{top}}}
(\hat{\mathcal W}^\B,(\mathcal W^\B)^{\otimes2\wedge}))$. Recall that 
$(\hat{\Delta}^{\mathcal W,\DR},\hat{\Delta}^{\mathcal W,\B})$ is a pair of elements of this pair of modules. 
By Lemma \ref{lem:1:6}, this gives rise to the stabilizer subbitorsor 
$_{\mathrm{Stab}_{\mathsf G^\DR(\mathbf k)}(\hat{\Delta}^{\mathcal W,\DR})}\mathrm{Stab}_{\mathsf G^\DR(\mathbf k)}(\hat{\Delta}^{\mathcal W,\DR},\hat{\Delta}^{\mathcal W,\B})_{\mathrm{Stab}_{\mathsf G^\B(\mathbf k)}(\hat{\Delta}^{\mathcal W,\B})}$ of $\mathsf G^{\DR,\B}(\mathbf k)$. \hfill\qed\medskip 

\subsection{$\mathsf{DMR}^{\DR,\B}(\mathbf k)$ as a subbitorsor of 
$\mathsf G^{\DR,\B}(\mathbf k)$}\label{sect:3:4}

\begin{lemdef}\label{lem:3:9:1705}
Set $\mathsf{DMR}^\B(\mathbf k):=\mathsf G^\B_{\mathrm{quad}}(\mathbf k)\cap
\mathsf{Stab}(\hat\Delta^{\mathcal M,\B})(\mathbf k)$. 
Then $\mathsf{DMR}^\B(\mathbf k)$ is a subgroup of $(\mathsf G^\B(\mathbf k),\circledast)$. 
\end{lemdef}
 
\proof This follows from Lemma \ref{lem:3:2:2204} and from the fact that 
$\mathsf{Stab}(\hat\Delta^{\mathcal M,\B})(\mathbf k)$ is a subgroup of $(\mathsf G^\B(\mathbf k),\circledast)$.
\hfill\qed\medskip 

In \cite{EF2}, Definition 2.12, we introduced the subsets 
$\mathsf{DMR}^\DR(\mathbf k)$ and $\mathsf{DMR}^{\DR,\B}(\mathbf k)$
of $\mathsf G^\DR(\mathbf k)$, and in Lemma 2.13, proved that 
$\mathsf{DMR}^\DR(\mathbf k)$ is a subgroup of  $\mathsf G^\DR(\mathbf k)$, and that 
$\mathsf{DMR}^{\DR,\B}(\mathbf k)$, equipped with the left action of this group, is a subtorsor
of $\mathsf G^{\DR,\B}(\mathbf k)$.

\begin{lem}\label{lem:3:10:2105}
The right action of $\mathsf G^\B(\mathbf k)$ on $\mathsf G^\DR(\mathbf k)$ restricts to a free and transitive action of 
$\mathsf{DMR}^\B(\mathbf k)$ on $\mathsf{DMR}^{\DR,\B}(\mathbf k)$, 
so that this right action equips $\mathsf{DMR}^{\DR,\B}(\mathbf k)$ with the structure of a 
subbitorsor of the bitorsor $\mathsf G^{\DR,\B}(\mathbf k)$.  
\end{lem}

\proof By Lemmas \ref{lem:1:3}, \ref{lemma:bitorsor:1} and \ref{subbitorsor:0105:bis}, the intersection of the bitorsors 
$\mathsf G^{\DR,\B}_{\mathrm{quad}}(\mathbf k)$ and $\mathsf{Stab}(\hat\Delta^{\mathcal M,\DR/\B})(\mathbf k)$ 
is a subbitorsor of $\mathsf G^{\DR,\B}(\mathbf k)$. The result then follows from 
$\mathsf{DMR}^\DR(\mathbf k)=\mathsf G^\DR_{\mathrm{quad}}(\mathbf k)\cap
\mathsf{Stab}(\hat\Delta^{\mathcal M,\DR})(\mathbf k)\subset\mathsf G^\DR(\mathbf k)$ and  
$\mathsf{DMR}^{\DR,\B}(\mathbf k)=\mathsf G^{\DR,\B}_{\mathrm{quad}}(\mathbf k)
\cap\mathsf{Stab}(\hat\Delta^{\mathcal M,\DR/\B})(\mathbf k)\subset\mathsf G^\DR(\mathbf k)$, which follows
from \cite{EF2}, Theorem 3.1, (a).   
\hfill\qed\medskip 

\subsection{$\mathsf{DMR}_\mu(\mathbf k)$ as a subbitorsor of 
$\mathsf G^{\DR,\B}(\mathbf k)$}\label{sect:3:5:3006}\label{sect:3:5}

\begin{lemdef}\label{lem:3:11:1705}
Set $\mathsf{DMR}^\B_0(\mathbf k):=\{g\in\mathcal G(\hat{\mathcal V}^\B)\ |\ 
(g|\mathrm{log}X_0)
=(g|\mathrm{log}X_1)=(g|\mathrm{log}X_0\mathrm{log}X_1)=0$ and $ 
(^\Gamma\!\mathrm{aut}_{(\mu,g)}^{\mathcal M,(10),\B})^{\otimes2}
\circ \hat\Delta^{\mathcal M,\B}=\hat\Delta^{\mathcal M,\B}\circ 
\ ^\Gamma\!\mathrm{aut}_{(\mu,g)}^{\mathcal M,(10),\B}
\}$. Then $\mathsf{DMR}^\B_0(\mathbf k)$ is a subgroup of 
$(\mathcal G(\hat{\mathcal V}^\B),\circledast)$. 
\end{lemdef}

\proof The map $(\mu,g)\mapsto\mu$ is a group morphism $\mathsf G^\B(\mathbf k)\to\mathbf k^\times$, whose kernel 
is equal to $(\mathcal G(\hat{\mathcal V}^\B),\circledast)$. Then $\mathsf{DMR}^\B_0(\mathbf k)$ is equal to the intersection of 
this kernel with $\mathsf{DMR}^\B(\mathbf k)$, and is therefore the intersection of two subgroups of $\mathsf G^\B(\mathbf k)$. 
\hfill\qed\medskip 

In \cite{Rac}, Racinet introduced the subsets $\mathsf{DMR}_0(\mathbf k)$ and $\mathsf{DMR}_\mu(\mathbf k)$
of $\mathcal G(\hat{\mathcal V}^\DR)$ for any $\mu\in\mathbf k$, and proved that $\mathsf{DMR}_0(\mathbf k)$ is a 
subgroup of  $(\mathcal G(\hat{\mathcal V}^\DR),\circledast)$, and that $\mathsf{DMR}_\mu(\mathbf k)$, equipped with the 
left action of this group, is a subtorsor of $\mathcal G(\hat{\mathcal V}^\DR)$, viewed as a trivial torsor over itself.  In 
\cite{EF2}, Lemma 2.14, we identified the torsor $\mathsf{DMR}_\mu(\mathbf k)$ with a subtorsor of 
$\mathsf {DMR}^{\DR,\B}(\mathbf k)$, viewed as a torsor under the left action of $\mathsf {DMR}^\DR(\mathbf k)$. 

\begin{lem}
The right action of $\mathsf{DMR}^\B(\mathbf k)$ on $\mathsf {DMR}^\DR(\mathbf k)$ restricts to a free and 
transitive action of $\mathsf{DMR}^\B_0(\mathbf k)$ on $\mathsf{DMR}_\mu(\mathbf k)$, 
so that this right action equips $\mathsf{DMR}_\mu(\mathbf k)$ with the structure of a 
subbitorsor of the bitorsor $\mathsf {DMR}^{\DR,\B}(\mathbf k)$. 
\end{lem}


\proof The maps $(\mu,g)\mapsto\mu$ set up a bitorsor morphism from
 $_{\mathsf{DMR}^\DR(\mathbf k)}\mathsf{DMR}^{\DR,\B}(\mathbf k)_{\mathsf{DMR}^B(\mathbf k)}$ to 
$_{\mathbf k^\times}\mathbf k^\times{}_{\mathbf k^\times}$. Then $_{\{1\}}\{\mu\}_{\{1\}}$ is a subbitorsor of the 
latter torsor. Its preimage is therefore a subbitorsor of the source bitorsor. The result follows from the identification of 
the preimage of kernel of $\mathsf{DMR}^\DR(\mathbf k)\to\mathbf k^\times$ (resp. 
$\mathsf{DMR}^\B(\mathbf k)\to\mathbf k^\times$) with $\mathsf{DMR}_0(\mathbf k)$ (resp. with 
$\mathsf{DMR}^\B_0(\mathbf k)$) and of the preimage of $\mu$ under the map 
$\mathsf{DMR}^{\DR,\B}(\mathbf k)\to\mathbf k^\times$ with 
$\mathsf{DMR}_\mu(\mathbf k)$, obtained in the proof of Lemma 2.14 in \cite{EF2}. \hfill\qed\medskip

\subsection{$\mathsf M(\mathbf k)$ as a subbitorsor of $\mathsf G^{\DR,\B}(\mathbf k)$}\label{sect:3:6}

It follows from \cite{Dr} that the subgroup $\mathsf{GT}^{\mathrm{op}}(\mathbf k)\subset\mathsf G^\B(\mathbf k)$
is the group attached to $\mathsf M(\mathbf k)$, viewed as a subtorsor of 
$\mathsf G^{\DR,\B}(\mathbf k)$ when equipped with the left action of $\mathsf{GRT}^{\mathrm{op}}(\mathbf k)$. 
It follows that $\mathsf M(\mathbf k)$, equipped with the commuting left and right actions of 
$\mathsf{GRT}(\mathbf k)^{\mathrm{op}}$ and $\mathsf{GRT}(\mathbf k)^{\mathrm{op}}$, is a subbitorsor of 
$\mathsf G^{\DR,\B}(\mathbf k)$. 

\subsection{Groups corresponding to some torsors and their interrelations}\label{sect:3:7}

\begin{thm}\label{thm:3:13:1705} 
(a) The group attached, in the sense of Lemma \ref{lem:gattt}, (a), to the subtorsor 
$
\mathsf{Stab}(\hat\Delta^{\mathcal W,\DR/\B})(\mathbf k)$ 
(resp., $
\mathsf{Stab}(\hat\Delta^{\mathcal M,\DR/\B})(\mathbf k)$, 
$
\mathsf{DMR}^{\DR,\B}(\mathbf k)$, 
$
\mathsf{DMR}_\mu(\mathbf k)$) of 
the torsor $\mathsf G^{\DR,\B}(\mathbf k)$ is the subgroup 
$\mathsf{Stab}(\hat\Delta^{\mathcal W,\B})(\mathbf k)$ 
(resp., $\mathsf{Stab}(\hat\Delta^{\mathcal M,\B})(\mathbf k)$, 
$\mathsf{DMR}^\B(\mathbf k)$, $\mathsf{DMR}^\B_0(\mathbf k)$) of $\mathsf G^\B(\mathbf k)$. 

(b) The following inclusions hold between subgroups of $\mathsf G^\B(\mathbf k)$: 
$$
\mathsf{GT}^{\mathrm{op}}(\mathbf k)\subset\mathsf{DMR}^\B(\mathbf k)
\subset
\mathsf{Stab}(\hat\Delta^{\mathcal M,\B})(\mathbf k)
\subset  
\mathsf{Stab}(\hat\Delta^{\mathcal W,\B})(\mathbf k). 
$$
\end{thm}

\proof (a) follows from the fact that for any bitorsor $_GX_H$, there is a unique group isomorphism 
$H\to\mathrm{Aut}_G(X)$ (see \S\ref{sect:1:2:1305}). Then (b) follows from Lemma \ref{lem:1:10:1305} from the present paper, 
Theorem 3.1 in \cite{EF2} and Lemma-Definition \ref{lem:3:9:1705} from the present paper. \hfill\qed\medskip 
 
\begin{rem}
\cite{DG} contains the construction of a $\mathbb Q$-linear neutral Tannakian category 
$\mathcal T:=\mathrm{MT}(\mathbb Z)_{\mathbb Q}$ of mixed Tate motives over $\mathbb Z$, 
equipped with `Betti' and `de Rham' fiber functors as in Lemma \ref{construction(**)}. 
It leads to the construction  $\mathsf{Iso}^{\otimes}(\omega_\B,
\omega_\DR)(\mathbf k)$, which is a bitorsor under the left (resp. right) actions of  
$\mathsf{Aut}^{\otimes}(\omega_\B)(\mathbf k)$ (resp. $\mathsf{Aut}^{\otimes}(\omega_\DR)(\mathbf k)$). 
There is a morphism from this bitorsor to the associator bitorsor $\mathsf M(\mathbf k)$
(see \cite{An}), which in its turn is equipped with morphisms to the bitorsors of Lemmas \ref{lemma:bitorsor:1}, 
\ref{subbitorsor:0105}, \ref{subbitorsor:0105:bis}, \ref{lem:3:10:2105}. This explains the namings `De Rham' and `Betti' and 
for the left and right sides of these bitorsors.  
\end{rem}

\subsection{Scheme-theoretic and Lie algebraic aspects}\label{sect:3:8}

The assignments 
\begin{equation}\label{list:gp:schemes}
\{\mathbb Q\operatorname{-algebras}\}\ni \mathbf k\mapsto 
\mathsf G^{\B}_{?}(\mathbf k), (\mathcal G(\hat{\mathcal V}^\B),\circledast), 
\mathsf{GT}(\mathbf k)^{\mathrm{op}},\mathsf{DMR}^\B_{??}(\mathbf k),
\mathsf{Stab}(\hat\Delta^{???,\DR})(\mathbf k),
\end{equation}
where ? stands for no index or quad, ?? stands for no index or 0, ??? stands for $\mathcal M$ or $\mathcal W$, 
are $\mathbb Q$-group schemes. Recall that the Lie algebra of a $\mathbb Q$-group scheme $\mathbf k\mapsto\mathsf G(\mathbf k)$ 
is defined as $\mathrm{Ker}(\mathsf G(\mathbb Q[\epsilon]/(\epsilon^2))\to \mathsf G(\mathbb Q))$; it is a $\mathbb Q$-Lie 
algebra. Let 
$$
\mathfrak g^\B_?, \quad
\mathfrak g^\B_0, \quad
\mathfrak{gt}^{\mathrm{op}}, \quad
\mathfrak{dmr}^\B_{??}, \quad
\mathfrak{stab}(\hat\Delta^{???,\B}),
$$
be the Lie algebras of the $\mathbb Q$-group schemes \eqref{list:gp:schemes}. 

Denote with an index $\mathbb Q$ the objects and morphisms of \S\ref{sect:2:1:1505} corresponding to 
$\mathbf k=\mathbb Q$. The primitive part of $\hat{\mathcal V}^\B_{\mathbb Q}
=(\mathbb QF_2)^\wedge$ with respect to $\hat\Delta^{\mathcal V,\B}_{\mathbb Q}$ is 
the complete Lie algebra $\mathrm{Lie} F_2(\mathbb Q)$, topologically generated by the free family 
$\mathrm{log}X_0$, $\mathrm{log}X_1$, where $\mathrm{log}$ is the logarithm map 
$1+(\hat{\mathcal V}^\B_{\mathbb Q})_+\to (\hat{\mathcal V}^\B_{\mathbb Q})_+$ . 

Let $D$ be the derivation of $\mathrm{Lie}F_2(\mathbb Q)$ defined by 
$\mathrm{log}X_i\mapsto \mathrm{log}X_i$, for $i=0,1$. For $x\in\mathrm{Lie}F_2(\mathbb Q)$, 
let $D_x$ be the derivation of $\mathrm{Lie}F_2(\mathbb Q)$  defined by 
$\mathrm{log}X_0\mapsto[x,\mathrm{log}X_0]$, $\mathrm{log}X_1\mapsto0$. 

\begin{lem}[see \cite{Dr}] (a) $\mathfrak g^\B$ is a complete Lie algebra, of which 
$\mathfrak{gt}^{\mathrm{op}}$ and $\mathfrak g^\B_0$ are complete Lie subalgebras. 

(b) $\mathfrak g^\B=\mathbb Q\oplus\mathrm{Lie}F_2(\mathbb Q)$, with Lie bracket 
$\langle(\nu,x),(\nu',x')\rangle=\nu D(x')-\nu' D(x)+D_x(x')-D_{x'}(x)-[x,x']$; $\mathfrak g^\B$ is the 
Lie subalgebra $\mathrm{Lie}F_2(\mathbb Q)$.  

(c) $\mathfrak{gt}^{\mathrm{op}}$ is the subspace of $\mathfrak g^\B$ defined 
by relations (5.5) to (5.7) in \cite{Dr}. 
\end{lem}

For $(\nu,x)\in\mathfrak g^\B$, set 
\begin{equation}\label{def:der:V:1705}
\mathrm{der}_{(\nu,x)}^{\mathcal V,(1),\B}:=\nu D+D_x\in\mathrm{End}_{\mathbb Q}
(\hat{\mathcal V}_{\mathbb Q}^\B), \quad 
\mathrm{der}_{(\nu,x)}^{\mathcal V,(10),\B}:=\nu D+D_x+r_x\in\mathrm{End}_{\mathbb Q}
(\hat{\mathcal V}_{\mathbb Q}^\B).  
\end{equation}
Then $\mathrm{der}_{(\nu,x)}^{\mathcal V,(1),\B}$ restricts to an endomorphism 
$\mathrm{der}_{(\nu,x)}^{\mathcal W,(1),\B}\in\mathrm{End}_{\mathbb Q}
(\hat{\mathcal W}_{\mathbb Q}^\B)$, and $\mathrm{der}_{(\nu,x)}^{\mathcal V,(10),\B}$ 
induces an endomorphism $\mathrm{der}_{(\nu,x)}^{\mathcal M,(10),\B}\in
\mathrm{End}_{\mathbb Q}(\hat{\mathcal M}_{\mathbb Q}^\B)$. 

For $x\in\mathrm{Lie}F_2(\mathbb Q)$, set 
$$
\gamma_x(t):=\sum_{n\geq1}(-1)^{n+1}(x|(\mathrm{log}X_0)^{n-1}\mathrm{log}X_1)t^n/n\in\mathbb Q[[t]] 
$$
and for $(\nu,x)\in\mathfrak g^\B$, set 
\begin{equation}\label{def:der:M:1705}
{}^\Gamma\!\!\mathrm{der}_{(\nu,x)}^{\mathcal W,(1),\B}:=
\mathrm{der}_{(\nu,x)}^{\mathcal W,(1),\B}
-\mathrm{ad}_{\gamma_x(-\mathrm{log}X_1)}\in\mathrm{End}_{\mathbb Q}
(\hat{\mathcal W}_{\mathbb Q}^\B), 
\end{equation}
\begin{equation}\label{def:der:W:11032020}
{}^\Gamma\!\!\mathrm{der}_{(\nu,x)}^{\mathcal M,(10),\B}:=\mathrm{der}_{(\nu,x)}^{\mathcal M,(10),\B}
-\ell_{\gamma_x(-\mathrm{log}X_1)}\in\mathrm{End}_{\mathbb Q}(\hat{\mathcal M}_{\mathbb Q}^\B).  
\end{equation}

\begin{lem} (a) $\mathfrak g^\B_{\mathrm{quad}}$, $\mathfrak{dmr}^\B_{??}$ and  
$\mathfrak{stab}(\hat\Delta^{???,\B})$ are complete Lie subalgebras of $\mathfrak g^\B$. 

(b) $\mathfrak g^\B_{\mathrm{quad}}=\{(\nu,x)\in\mathfrak g^\B|(x|\mathrm{log}X_0)=(x|\mathrm{log}X_1)=0, 
\nu=12(x|\mathrm{log}X_0\mathrm{log}X_1)\}$.

(c) $\mathfrak{stab}(\hat\Delta^{\mathcal W,\B})$ is the set of elements $(\nu,x)\in\mathfrak g^\B$ 
such that $({}^\Gamma\!\!\mathrm{der}^{\mathcal W,(1),\B}_{(\nu,x)}\otimes \mathrm{id}+\mathrm{id}\otimes 
{}^\Gamma\!\!\mathrm{der}^{\mathcal W,(1),\B}_{(\nu,x)})\circ\hat\Delta^{\mathcal W,\B}_{\mathbb Q}
=\hat\Delta^{\mathcal W,\B}_{\mathbb Q}\circ\ {}^\Gamma\!\!\mathrm{der}^{\mathcal W,(1),\B}_{(\nu,x)}$ (equality in 
$\mathrm{Hom}_{\mathbb Q}(\hat{\mathcal W}_{\mathbb Q}^\B,({\mathcal W}_{\mathbb Q}^\B)^{\otimes2,\wedge})$). 

(d) $\mathfrak{stab}(\hat\Delta^{\mathcal M,\B})$ is the set of elements $(\nu,x)\in\mathfrak g^\B$ such that 
$({}^\Gamma\!\!\mathrm{der}^{\mathcal M,(10),\B}_{(\nu,x)}\otimes \mathrm{id}+\mathrm{id}\otimes 
{}^\Gamma\!\!\mathrm{der}^{\mathcal M,(10),\B}_{(\nu,x)})\circ\hat\Delta^{\mathcal M,\B}_{\mathbb Q}
=\hat\Delta^{\mathcal M,\B}_{\mathbb Q}\circ\ {}^\Gamma\!\!\mathrm{der}^{\mathcal M,(10),\B}_{(\nu,x)}$ (equality in 
$\mathrm{Hom}_{\mathbb Q}(\hat{\mathcal M}_{\mathbb Q}^\B,({\mathcal M}_{\mathbb Q}^\B)^{\otimes2,\wedge})$). 

(e) $\mathfrak{dmr}^\B=\mathfrak g^\B_{\mathrm{quad}}\cap 
\mathfrak{stab}(\hat\Delta^{\mathcal M,\B})$ and $\mathfrak{dmr}^\B_0=\mathfrak g^\B_0
\cap \mathfrak g^\B_{\mathrm{quad}}\cap \mathfrak{stab}(\hat\Delta^{\mathcal M,\B})$.
\end{lem}

\proof (b), (c), (d) are obtained by linearization. (e) follows from the equalities $\mathsf{DMR}^\B(\mathbf k)
=\mathsf G^\B_{\mathrm{quad}}(\mathbf k)\cap\mathsf{Stab}(\hat\Delta^{\mathcal M,\DR/\B})(\mathbf k)$ and 
$\mathsf{DMR}^\B_0(\mathbf k)=\mathsf{DMR}^\B(\mathbf k)\cap(\mathcal G(\hat{\mathcal V}^\B),\circledast)$
(see Lemmas \ref{lem:3:9:1705}, \ref{lem:3:11:1705}). (a) follows from (b)-(e). \hfill\qed\medskip 

\begin{rem} The endomorphisms $\mathrm{der}_{(\nu,x)}^{?}$, $^{\Gamma}\!\mathrm{der}_{(\nu,x)}^{??}$ are Lie 
algebraic analogues of the automorphisms $\mathrm{aut}_{(\mu,g)}^{?}$, $^{\Gamma}\!\mathrm{aut}_{(\mu,g)}^{??}$,
where $?$ (resp. $??$) takes the values $(\mathcal X,\B,\alpha)$, with $(\mathcal X,\alpha)$ in $\{(\mathcal V,(1)),(\mathcal V,(10)),
(\mathcal W,(1)),(\mathcal M,(10))\}$ (resp. in $\{(\mathcal W,(1)),(\mathcal M,(10))\}$), in particular, any of the maps taking 
$(\nu,x)$ to one of these endomorphisms is a representation of $\mathfrak g^\B$. 
\end{rem}

\begin{cor} The following inclusions between Lie subalgebras of $\mathfrak g^\B$ hold: 
(a) $\mathfrak{gt}^{\mathrm{op}}\subset\mathfrak{dmr}^\B$, 
(b) $\mathfrak{stab}(\hat\Delta^{\mathcal M,\B})\subset\mathfrak{stab}(\hat\Delta^{\mathcal W,\B})$. 
\end{cor}

\proof This follows directly from Theorem \ref{thm:3:13:1705}. \hfill\qed\medskip 

\subsection{Relation with Hopf algebra and coalgebra isomorphism bitorsors}\label{rwhacaib}

Recall that a {\it module-coalgebra} over a (topological) $\mathbf k$-Hopf algebra $(W,\Delta_W)$ is a  coassociative counital (topological)
$\mathbf k$-coalgebra $(M,\Delta_M)$, equipped with a map $W\otimes M\to M$ which both defines an action of $W$ (viewed as an associative 
$\mathbf k$-algebra) on $M$ (viewed as a $\mathbf k$-module) and a morphism of counital $\mathbf k$-coalgebras ($W\otimes M$ being equipped 
with the tensor product coalgebra structure). Denote by $\mathbf k\operatorname{-HAMC}$ the category of pairs $((W,\Delta_W),(M,\Delta_M))$
of a $\mathbf k$-Hopf algebra and a module-coalgebra over it. It fits in a commutative diagram of forgetful functors 
\begin{equation}\label{six:categories}
    \xymatrix{\mathbf k\operatorname{-coalg}\ar[d]&\ar[l]\mathbf k\operatorname{-HAMC} \ar[r]\ar[d]& \mathbf k\operatorname{-Hopf}\ar[d]\\ \mathbf 
k\operatorname{-mod}&\ar[l]\mathbf k\operatorname{-alg-mod}\ar[r]& \mathbf k\operatorname{-alg}}
\end{equation}
where $\mathbf k$-Hopf (resp. $\mathbf k$-coalg, $\mathbf k$-mod) is the category of $\mathbf k$-Hopf algebras (resp. $\mathbf k$-coalgebras, 
$\mathbf k$-modules), in which the vertical functors are faithful. A pair of isomorphic objects 
in $\mathbf k\operatorname{-HAMC}$ then induces a commutative diagram of isomorphism bitorsors, where the vertical maps are bitorsor inclusions. 

If $\mathcal C$ is a category and $\omega\mapsto X^\omega$ is a map $\{\B,\DR\}\to\mathrm{Ob}(\mathcal C)$, let us denote by 
$\mathrm{Iso}_{\mathcal C}(X^{\DR/\B})$ the bitorsor of isomorphisms $\mathrm{Iso}_{\mathcal C}(X^\B,X^{\DR})$. If $X^\omega$ is given by 
a tuple $(A^\omega,B^\omega,\ldots)$, we set $(A,B,\ldots)^\omega:=(A^\omega,B^\omega,\ldots)$. 

\begin{lem}\label{lem:cartesianity}
Let $((W^\omega,\Delta_W^\omega),(M^\omega,\Delta_M^\omega))$, $\omega\in\{\B,\DR\}$ be a pair of isomorphic objects in 
$\mathbf k\operatorname{-HAMC}$ such that $M^\B$ is free of rank one over $W^\B$ with generator $1_M^\B$ and such that 
$\Delta_M^\B(1_M^\B)\in((W^\B)^{\otimes2})^\times\cdot(1_M^\B)^{\otimes2}$. Then 
the commutative diagram of bitorsors 
$$
\xymatrix{\mathrm{Iso}_{\mathbf k\operatorname{-coalg}}((M,\Delta_M)^{\DR/\B})\ar[d]&\ar[l]\ar[d]
\mathrm{Iso}_{\mathbf k\operatorname{-HAMC}}(((W,\Delta_W),(M,\Delta_M))^{\DR/\B})
\\\mathrm{Iso}_{\mathbf k\operatorname{-mod}}(M^{\DR/\B})&\mathrm{Iso}_{\mathbf k\operatorname{-alg-mod}}((W,M)^{\DR/\B})\ar[l]}
$$
induced by the left square of \eqref{six:categories} is Cartesian (see Definition \ref{def:Cartesian}).  
\end{lem}

\begin{proof} The argument follows that of \cite{EF2}, \S3.4. 
Let $(c_W,c_M)\in \mathrm{Iso}_{\mathbf k\operatorname{-alg-mod}}((W,M)^{\DR/\B})$ be such that $c_M\in
 \mathrm{Iso}_{\mathbf k\operatorname{-coalg}}((M,\Delta_M)^{\DR/\B})$. Then for $w\in W^\B$, 
\begin{align*}
& c_W^{\otimes2}(\Delta_W^\B(w))\cdot 
c_M^{\otimes2}(\Delta_M^\B(1_M^\B))=
c_M^{\otimes2}(\Delta_W^\B(w)\cdot \Delta_M^\B(1_M^\B))=
c_M^{\otimes2}(\Delta_M^\B(w\cdot 1_M^\B))=\Delta_M^\DR(c_M(w\cdot 
1_M^\B))
\\ & =
\Delta_M^\DR(c_W(w)\cdot c_M(1^\B_M))=
\Delta_W^\DR(c_W(w))\cdot \Delta_M^\DR(c_M(1^\B_M))=
\Delta_W^\DR(c_W(w))\cdot c_M^{\otimes2}(\Delta_M^\B(1_M^\B))
\end{align*}
by the various axioms. Then $c_M^{\otimes2}(\Delta_M^\B(1_M^\B))\in c_M^{\otimes2}(((W^\B)^{\otimes2})^\times\cdot(1_M^\B)^{\otimes2})
=c_W^{\otimes2}(((W^\B)^{\otimes2})^\times)\cdot c_M(1_M^\B)^{\otimes2}\subset ((W^\DR)^{\otimes2})^\times\cdot c_M(1_M^\B)^{\otimes2}$
and which together with the fact that $c_M(1_M^\B)$ is a generator of $M^\DR$ as a free $W^\DR$-module implies 
$c_W^{\otimes2}(\Delta_W^\B(w))=\Delta_W^\DR(c_W(w))$ therefore 
therefore $c_W^{\otimes2}\circ\Delta_W^\B\circ c_W^{-1}=\Delta_W^\DR$. 
It follows that the Hopf algebra structure on $W^\B$ pulled back from the Hopf algebra structure of $W^\DR$ has coproduct $\Delta_W^\B$. 
The uniqueness of the counit in a coalgebra in a symmetric monoidal category (proved by dualyzing the standard argument proving the 
uniqueness of the unit in an algebra) and of the antipode in a bialgebra in such a category (based on the argument of \cite{Sw}, p. 71) 
then imply that the pull-back on $W^\B$ of the Hopf algebra structure of $W^\DR$ is the Hopf algebra structure of $W^\B$, therefore 
$c_W\in\mathrm{Iso}_{\mathbf k\operatorname{-Hopf}}(W^{\DR/\B})$ so 
$(c_W,c_M)\in \mathrm{Iso}_{\mathbf k\operatorname{-HAMC}}(((W,\Delta_W),(M,\Delta_M))^{\DR/\B})$. 
\end{proof}

\begin{lem}\label{lem:pre:diag}
(a) For $\omega\in\{\B,\DR\}$, $((\hat{\mathcal W}^\omega,\Delta^{\mathcal W,\omega}),(\hat{\mathcal M}^\omega,\Delta^{\mathcal M,\omega}))$ is an 
object in $\mathbf k\operatorname{-HAMC}$. 

(b) The map $(\mu,\Phi)\mapsto(^\Gamma\mathrm{comp}^{\mathcal W,(1)}_{(\mu,\Phi)},^\Gamma\mathrm{comp}^{\mathcal M,(10)}_{(\mu,\Phi)})$ 
defines a bitorsor morphism $\mathsf G^{\DR,\B}(\mathbf k)\to\mathrm{Iso}_{\mathbf k\operatorname{-alg-mod}}
((\hat{\mathcal W},\hat{\mathcal M})^{\DR/\B})$. Its composition with the bitorsor morphism to 
$\mathrm{Iso}_{\mathbf k\operatorname{-alg}}(\hat{\mathcal W}^{\DR/\B})$ (resp. 
$\mathrm{Iso}_{\mathbf k\operatorname{-mod}}(\hat{\mathcal M}^{\DR/\B})$) is 
$(\mu,\Phi)\mapsto\ ^\Gamma\mathrm{comp}^{\mathcal M,(1)}_{(\mu,\Phi)}$. 
\end{lem}

\begin{proof} The $\omega=\DR$ part of (a) follows from \S\S1.1 and 1.2 in \cite{EF1}, and its $\omega=\B$ part follows from 
\S\S2.1 and 2.4 in {\it loc. cit}.  (b) follows from \cite{EF2}, Lemmas 1.21 to 1.24. 
\end{proof}

Lemma \ref{lem:pre:diag} leads to a diagram of bitorsors
\begin{equation}\label{seven:bitorsors}
    \xymatrix{\mathrm{Iso}_{\mathbf k\operatorname{-coalg}}
((\hat{\mathcal M},\hat\Delta^{\mathcal M})^{\DR/\B})\ar@{^{(}->}[d]&\ar[l]
\mathrm{Iso}_{\mathbf k\operatorname{-HAMC}}(((\hat{\mathcal W},\hat\Delta^{\mathcal W}),
(\hat{\mathcal M},\hat\Delta^{\mathcal M}))^{\DR/\B})\ar[r]\ar@{^{(}->}[d]& 
\mathrm{Iso}_{\mathbf k\operatorname{-Hopf}}((\hat{\mathcal W},\hat\Delta^{\mathcal W})^{\DR/\B})\ar@{^{(}->}[d]\\ 
\mathrm{Iso}_{\mathbf k\operatorname{-mod}}(\hat{\mathcal M}^{\DR/\B})&\ar[l]
\mathrm{Iso}_{\mathbf k\operatorname{-alg-mod}}((\hat{\mathcal W},\hat{\mathcal M})^{\DR/\B})\ar[r]&
\mathrm{Iso}_{\mathbf k\operatorname{-alg}}(\hat{\mathcal W}^{\DR/\B})\\ & \mathsf G^{\DR,\B}(\mathbf k)\ar[u]& }
\end{equation}
which, by denoting each of the bitorsors $\mathrm{Iso}_{\mathbf k\operatorname{-}\mathcal C}(X^{\DR/\B})$ of \eqref{seven:bitorsors}
by $T_{\mathcal C}$ leads to  a diagram of preimage subbitorsors of $\mathsf G^{\DR,\B}(\mathbf k)$: 
\begin{equation}\label{diagofsubbitorsors}
\mathsf G^{\DR,\B}(\mathbf k)\times_{T_{\operatorname{mod}}}
T_{\operatorname{coalg}}\supset 
\mathsf G^{\DR,\B}(\mathbf k)\times_{T_{\operatorname{alg-mod}}}
T_{\operatorname{HAMC}}\subset
\mathsf G^{\DR,\B}(\mathbf k)\times_{T_{\operatorname{alg}}}
T_{\operatorname{Hopf}}
\end{equation}


\begin{lem}\label{lemma:identifications:subbitorsors}
(a) The subbitorsor $\mathsf G^{\DR,\B}(\mathbf k)\times_{T_{\operatorname{mod}}}T_{\operatorname{coalg}}
$ of $\mathsf G^{\DR,\B}(\mathbf k)$ coincides with 
$\mathsf{Stab}(\hat\Delta^{\mathcal M,\DR/\B})(\mathbf k)$. 

(b) The subbitorsor $\mathsf G^{\DR,\B}(\mathbf k)\times_{T_{\operatorname{alg}}}T_{\operatorname{Hopf}}$ 
of $\mathsf G^{\DR,\B}(\mathbf k)$ coincides with 
$\mathsf{Stab}(\hat\Delta^{\mathcal W,\DR/\B})(\mathbf k)$. 

(c) The subbitorsors $\mathsf G^{\DR,\B}(\mathbf k)\times_{T_{\operatorname{mod}}}T_{\operatorname{coalg}}$ and 
$\mathsf G^{\DR,\B}(\mathbf k)\times_{T_{\operatorname{alg-mod}}}T_{\operatorname{HAMC}}$
of  $\mathsf G^{\DR,\B}(\mathbf k)$ are equal. 
\end{lem}

\begin{proof} (a) Let $(\mu,\Phi)\in\mathsf G^{\DR,\B}(\mathbf k)$. Then $(\mu,\Phi)
\in\mathsf G^{\DR,\B}(\mathbf k)\times_{T_{\operatorname{mod}}}T_{\operatorname{coalg}}$ iff 
$^\Gamma\mathrm{comp}_{(\mu,\Phi)}^{\mathcal M,(10)} : (\hat{\mathcal M}^\B,\hat\Delta^{\mathcal M,\B})
 \to(\hat{\mathcal M}^\DR,\hat\Delta^{\mathcal M,\DR})$ is a coalgebra isomorphism. This implies the equality 
 $(^\Gamma\mathrm{comp}_{(\mu,\Phi)}^{\mathcal M,(10)})^{\otimes2}\circ\hat\Delta^{\mathcal M,\B}
 \circ(^\Gamma\mathrm{comp}_{(\mu,\Phi)}^{\mathcal M,(10)})^{-1}=\hat\Delta^{\mathcal M,\DR}$. 
 Therefore $(\mu,\Phi)\in\mathsf{Stab}(\hat\Delta^{\mathcal M,\DR/\B})$.  
Conversely, if $(\mu,\Phi)\in\mathsf{Stab}(\hat\Delta^{\mathcal M,\DR/\B})$, then the counital coalgebra structure on 
$\hat{\mathcal M}^\B$ pulled back from that of $\hat{\mathcal M}^\DR$ has coproduct $\hat\Delta^{\mathcal M,\B}$. 
By the uniqueness of the counit in a coalgebra, the counit of the pulled back coalgebra structure necessarily coincides with that of 
$\hat{\mathcal M}^\B$, therefore $^\Gamma\mathrm{comp}_{(\mu,\Phi)}^{\mathcal M,(10)}$ is a coalgebra isomorphism, so that 
$(\mu,\Phi)\in\mathsf G^{\DR,\B}(\mathbf k)\times_{T_{\operatorname{mod}}}T_{\operatorname{coalg}}$.  

(b) The proof is similar to (a), using in addition to the uniqueness of the counit in a coalgebra, the uniqueness of the antipode in a 
bialgebra (see the argument of Lemma \ref{lem:cartesianity}). 

(c) Let  $(\mu,\Phi)\in\mathsf G^{\DR,\B}(\mathbf k)$. Then $(\mu,\Phi)\in\mathsf G^{\DR,\B}(\mathbf 
k)\times_{T_{\operatorname{alg-mod}}}T_{\operatorname{HAMC}}$ iff 
$(^\Gamma\mathrm{comp}_{(\mu,\Phi)}^{\mathcal W,(1)},^\Gamma\mathrm{comp}_{(\mu,\Phi)}^{\mathcal M,(10)})$
belongs to $T_{\operatorname{HAMC}}$. By Lemma \ref{lem:cartesianity}, this condition is equivalent to
the conjunction of 
$$
(^\Gamma\mathrm{comp}_{(\mu,\Phi)}^{\mathcal W,(1)},^\Gamma\mathrm{comp}_{(\mu,\Phi)}^{\mathcal M,(10)})\in T_{\operatorname{alg-mod}}
$$
and $^\Gamma\mathrm{comp}_{(\mu,\Phi)}^{\mathcal W,(1)}\in T_{\operatorname{Hopf}}$. 
Since $(^\Gamma\mathrm{comp}_{(\mu,\Phi)}^{\mathcal W,(1)},^\Gamma\mathrm{comp}_{(\mu,\Phi)}^{\mathcal M,(10)})$ belongs to 
$T_{\operatorname{alg-mod}}$, this conjunction is equivalent to $^\Gamma\mathrm{comp}_{(\mu,\Phi)}^{\mathcal W,(1)}\in T_{\operatorname{Hopf}}$, 
i.e. $(\mu,\Phi)\in\mathsf G^{\DR,\B}(\mathbf 
k)\times_{T_{\operatorname{alg}}}T_{\operatorname{Hopf}}$. 
\end{proof}

The combination of Lemma \ref{lemma:identifications:subbitorsors} and \eqref{diagofsubbitorsors} then enables one to recover the   
inclusion $\mathsf{Stab}(\hat\Delta^{\mathcal M,\DR/\B}) \subset \mathsf{Stab}(\hat\Delta^{\mathcal W,\DR/\B})$ of subbitorsors of 
$\mathsf G^{\DR,\B}(\mathbf k)$ (see Theorem \ref{thm:3:13:1705}). 

One the other hand, it follows from Lemma \ref{lemma:identifications:subbitorsors} (a), (b)  that the bitorsor inclusion 
$\mathsf M(\mathbf k)\subset \mathsf{Stab}(\hat\Delta^{\mathcal M,\DR})(\mathbf k)$  (see Theorem \ref{thm:3:13:1705}), which 
essentially relies on the geometric interpretation of $\Delta^{\mathcal M,\DR}$ (see \cite{EF1}), is equivalent to the existence 
of a bitorsor morphism $\mathsf M(\mathbf k)\to\mathrm{Iso}_{\mathbf k\operatorname{-HAMC}}(((\hat{\mathcal W},\hat\Delta^{\mathcal W}),
(\hat{\mathcal M},\hat\Delta^{\mathcal M}))^{\DR/\B})$ such that the 
diagram 
$$
\xymatrix{\mathsf M(\mathbf k)\ar[r]\ar@{^{(}->}[d] & \mathrm{Iso}_{\mathbf k\operatorname{-HAMC}}(((\hat{\mathcal W},\hat\Delta^{\mathcal W}),
(\hat{\mathcal M},\hat\Delta^{\mathcal M}))^{\DR/\B})\ar@{^{(}->}[d]\\
\mathsf G^{\DR,\B}(\mathbf k)\ar[r]& \mathrm{Iso}_{\mathbf k\operatorname{-alg-mod}}((\hat{\mathcal W},
\hat{\mathcal M})^{\DR/\B})}
$$
commutes. 




\section{Equivalent definitions of $\mathsf{DMR}^\B(\mathbf k)$ and its Lie algebra}\label{sect:4:0307}

In this section, we prove that the group $\mathsf{DMR}^\B(\mathbf k)$
can be given by a definition, alternative to Lemma-Definition \ref{lem:3:9:1705} (\S\ref{sect:4:1:0307}, 
Theorem \ref{thm:4:5:1606}). This result should be viewed as a `Betti' counterpart to 
\cite{EF0}, Theorem 1.2 and \cite{EF2}, Lemma 3.8. In \S\ref{sect:4:2:0307}, we draw the Lie algebraic
consequence of this result, which is a counterpart of \cite{EF0}, Theorem 3.1 and 
\cite{EF2}, Corollary 3.12, (b). 

\subsection{Equivalent definition of $\mathsf{DMR}^\B(\mathbf k)$}\label{sect:4:1:0307}

The following results are analogues of Lemmas 1.17 and 3.5 from \cite{EF2}. 

\begin{lem}\label{fact:A}
The map  $(\mu,g)\mapsto\ ^\Gamma\!\mathrm{aut}^{\mathcal M,(10),\B}_{(\mu,g)}$  exhibits the following compatibility 
with the map $(\lambda,\varphi)\mapsto\ ^\Gamma\!\mathrm{comp}^{\mathcal M,(10)}_{(\lambda,\varphi)}$  and with the 
right action of $\mathsf G^\B(\mathbf k)$ on $\mathsf G^\DR(\mathbf k)$ (see Definition \ref{def:2:12:1244}): for 
$(\lambda,\varphi)\in\mathsf G^\DR(\mathbf k)$ and $(\mu,g)\in\mathsf G^\B(\mathbf k)$, one has 
$^\Gamma\!\mathrm{comp}^{\mathcal M,(10)}_{(\lambda,\varphi)\circledast\mathrm{iso}^G(\mu,g)}
=\ ^\Gamma\!\mathrm{comp}^{\mathcal M,(10)}_{(\lambda,\varphi)}\circ
\ ^\Gamma\!\mathrm{aut}^{\mathcal M,(10),\B}_{(\mu,g)}$.  
\end{lem}

\proof One combines the identity $^\Gamma\!\mathrm{aut}^{\mathcal M,(10),\B}_{(\mu,g)}=
(\mathrm{iso}^{\mathcal M})^{-1}\circ\  ^\Gamma\!\mathrm{aut}^{\mathcal M,(10),\DR}_{\mathrm{iso}^G(\mu,g)}
\circ \mathrm{iso}^{\mathcal M}$ with the identity (1.7.2) from the proof of Lemma 1.17 in \cite{EF2} and with the group 
morphism property of the map $(\mathsf G^\DR(\mathbf k),\circledast)\to\mathrm{Aut}_{\mathbf k\operatorname{-mod}}
(\hat{\mathcal M}^\DR)$, $(\mu,g)\mapsto\  ^\Gamma\!\mathrm{aut}^{\mathcal M,(10),\DR}_{(\mu,g)}$ proved there. 
\hfill\medskip 

\begin{lem}\label{fact:B}
Let $(\mu,g)\in\mathsf G^\B(\mathbf k)$, then  
$^\Gamma\!\mathrm{aut}^{\mathcal M,(10),\B}_{(\mu,g)}(1_\B) 
=(\Gamma_g(-\mathrm{log}X_1)^{-1}\cdot g)\cdot 1_\B$. 
\end{lem}

\proof  One has $^\Gamma\!\mathrm{aut}^{\mathcal V,(10),\B}_{(\mu,g)}(1_\B)=g$
by \eqref{aut:V:10:1705}, which by Lemma-Definition \ref{lemma:2:5:12042020} implies 
$\mathrm{aut}^{\mathcal M,(10),\B}_{(\mu,g)}(1_\B)=g\cdot 1_\B$. Then 
Lemma-Definition \ref{lem:2:9:0105} implies the result. \hfill\qed\medskip 

\begin{lem}\label{fact:C}
Let $(\lambda,\varphi)\in\mathsf M(\mathbf k)$. If $(\mu,g)\in\mathsf G^\B(\mathbf k)$, then the topological $\mathbf k$-module 
isomorphism
$^\Gamma\!\mathrm{comp}^{\mathcal M,(10)}_{(\lambda,\varphi)}:\hat{\mathcal M}^\B\to\hat{\mathcal M}^\DR$ takes 
$(\Gamma_g^{-1}(-\mathrm{log}X_1)\cdot g)\cdot 1_\B$ 
to $(\Gamma_{\varphi\circledast(\lambda\bullet\mathrm{iso}^{\mathcal G}(g))}^{-1}(-e_1)\cdot 
(\varphi\circledast(\lambda\bullet\mathrm{iso}^{\mathcal G}(g))))\cdot 1_\DR$ (with the notation as in \cite{EF2}, \S1.6). 
\end{lem}

\proof One has  
\begin{align*}
& (\Gamma_{\varphi\circledast(\lambda\bullet\mathrm{iso}^{\mathcal G}(g))}^{-1}(-e_1)\cdot 
(\varphi\circledast(\lambda\bullet\mathrm{iso}^{\mathcal G}(g))))\cdot 1_\DR
=\ ^\Gamma\!\mathrm{comp}^{\mathcal M,(10)}_{(\lambda,\varphi)\circledast\mathrm{iso}^G(\mu,g)}(1_\B) 
\\ & =\ ^\Gamma\!\mathrm{comp}^{\mathcal M,(10)}_{(\lambda,\varphi)}\circ 
\ ^\Gamma\!\mathrm{aut}^{\mathcal M,(10),\B}_{(\mu,g)}(1_\B) 
=\ ^\Gamma\!\mathrm{comp}^{\mathcal M,(10)}_\varphi((\Gamma_g^{-1}(-\mathrm{log}X_1)\cdot g)\cdot 1_\B),  
\end{align*}
where the first equality follows from applying Lemma 3.5 from \cite{EF2} to $(\lambda,\varphi)\circledast 
\mathrm{iso}^G(\mu,g)=(\lambda\mu,\varphi\circledast(\lambda\bullet\mathrm{iso}^{\mathcal G}(g)))$ (see \cite{EF2}, (1.6.2)), 
the second equality follows from applying Lemma \ref{fact:A} to $1_\B$, and the third one follows from applying 
Lemma \ref{fact:B} to $(\mu,g)$.  \hfill\qed\medskip 

\begin{lem}\label{fact:D}
Let $(\lambda,\varphi)\in\mathsf M(\mathbf k)$. The topological $\mathbf k$-module isomorphism
$^\Gamma\!\mathrm{comp}^{\mathcal M,(10)}_{(\lambda,\varphi)}:\hat{\mathcal M}^\B\stackrel{\sim}{\to} 
\hat{\mathcal M}^\DR$ restricts to a bijection $\mathcal G(\hat{\mathcal M}^\B)\stackrel{\sim}{\to}
\mathcal G(\hat{\mathcal M}^\DR)$, where $\mathcal G$ denotes the group-like parts of $\hat{\mathcal M}^\B$ 
(resp. $\hat{\mathcal M}^\DR$) for $\hat{\Delta}^{\mathcal M,\B}$ (resp. $\hat{\Delta}^{\mathcal M,\DR}$). 
\end{lem}

\proof This follows from the fact that $^\Gamma\!\mathrm{comp}^{\mathcal M,(10)}_{(\lambda,\varphi)}$ 
intertwines $\hat\Delta^{\mathcal M,\B}$ and $\hat\Delta^{\mathcal M,\DR}$, see \cite{EF1}, 
Theorem 3.1 (a), (b) and Definition 2.17. \hfill\qed\medskip 

\begin{thm}\label{thm:4:5:1606}
$\mathsf{DMR}^\B(\mathbf k)$ is equal to the set of $\mathsf G^\B(\mathbf k)$ of pairs $(\mu,g)$ satisfying the following conditions: 

$(1)$ $(\Gamma_g^{-1}(-\mathrm{log}X_1)\cdot g)\cdot 1_\B \in\mathcal G(\hat{\mathcal M}^\B)$, where $\Gamma_g$ is 
as in Definition \ref{def:Gamma:Betti}; 

$(2)$ $(g|\mathrm{log}X_0)=(g|\mathrm{log}X_1)=0$, $\mu^2=1+24(g|\mathrm{log}X_0\mathrm{log}X_1)$. 
\end{thm} 

\proof It follows from \cite{Dr}, Proposition 5.3 that $\mathsf M_1(\mathbb Q)=\mathsf M(\mathbb Q)\cap
(\{1\}\times\mathcal G(\hat{\mathcal V}^\DR))$ is nonempty. Let us denote by $(1,\varphi)\in\mathsf M(\mathbf k)$ the image 
of an element of $\mathsf M_1(\mathbb Q)$. Then, by \cite{EF2}, Theorem 3.1 (a), 
$(1,\varphi)\in\mathsf{DMR}^{\DR,\B}(\mathbf k)$. 

Let $(\mu,g)\in G^\B(\mathbf k)$. The right torsor property of $\mathsf{DMR}^{\DR,\B}(\mathbf k)$ under the action of 
$\mathsf{DMR}^\B(\mathbf k)$ as in Definition \ref{def:2:12:1244} (see Lemma \ref{lem:3:10:2105}) implies the equivalence
$$
((\mu,g)\in\mathsf{DMR}^\B(\mathbf k))\iff ((1,\varphi)\circledast\mathrm{iso}^G(\mu,g)\in
\mathsf{DMR}^{\DR,\B}(\mathbf k));  
$$
the equality $(1,\varphi)\circledast\mathrm{iso}^G(\mu,g)=(\mu,\varphi\circledast\mathrm{iso}^{\mathcal G}(g))$ 
(see \cite{EF2}, (1.6.2)) and \cite{EF2}, Definition 2.12 implies the equivalence 
$$
((1,\varphi)\circledast\mathrm{iso}^G(\mu,g)\in
\mathsf{DMR}^{\DR,\B}(\mathbf k)) \iff (\text{(a) and (b)}),
$$
where (a), (b) are the following statements:  

(a) $(\Gamma^{-1}_{\varphi\circledast\mathrm{iso}^{\mathcal G}(g)}(-e_1)\cdot
(\varphi\circledast\mathrm{iso}^{\mathcal G}(g)))\cdot 1_\DR\in 
\mathcal G(\hat{\mathcal M}^\DR)$ 

(b) $(1,\varphi)\circledast\mathrm{iso}^G(\mu,g)\in\mathsf G^{\DR,\B}_{\mathrm{quad}}(\mathbf k)$.

The fact that $\mathsf G^{\DR,\B}_{\mathrm{quad}}(\mathbf k)_{\mathsf G^\B_{\mathrm{quad}}(\mathbf k)}$ 
is a subtorsor of the right torsor $\mathsf G^{\DR}(\mathbf k)_{\mathsf G^\B(\mathbf k)}$ and the 
inclusions $(1,\varphi)\in\mathsf{DMR}^{\DR,\B}(\mathbf k)\subset
\mathsf G^{\DR,\B}_{\mathrm{quad}}(\mathbf k)$ imply the equivalence 
$$
(\text{b})\iff ((\mu,g)\in \mathsf G_{\mathrm{quad}}^\B(\mathbf k)). 
$$
Lemmas \ref{fact:C} and \ref{fact:D} imply the equivalence 
$$
(\text{a})\iff[(\Gamma_g^{-1}(-e_1)\cdot g)\cdot 1_\B\in\mathcal G(\hat{\mathcal M}^\B)]. 
$$
The above equivalences combine into $(g\in\mathsf{DMR}^\B(\mathbf k))\iff 
[((\mu,g)\in\mathsf G_{\mathrm{quad}}^\B(\mathbf k))$ and $(\Gamma_g^{-1}(-e_1)\cdot g)\cdot 1_\B\in
\mathcal G(\hat{\mathcal M}^\B)]$, which by Definition \ref{def:G:B:quad:1321} yields the announced equivalence. 
\hfill\qed\medskip  

\begin{rem}
Be \cite{EF1} Lemma 9.5 and Remark 9.6, we have
$$
\Gamma_\varphi(t)\Gamma_\varphi(-t)=\frac{\mu t}{e^{\mu t/2}-e^{-\mu t/2}}
$$
for $(\mu, \varphi)\in\mathsf{DMR}^{\DR,\B}(\mathbf k)$.
While for $(\mu', \varphi'):=(\mu, \varphi)\cdot (\lambda,g)
\in\mathsf{DMR}^{\DR,\B}(\mathbf k)$
with $(\mu, \varphi)\in\mathsf{DMR}^{\DR,\B}(\mathbf k)$ and 
$(\lambda,g)\in\mathsf{DMR}^{\B}(\mathbf k)$,
we have $\mu'=\mu\lambda$ and
$(\varphi'|e_0^{k-1}e_1)=(\varphi|e_0^{k-1}e_1)+\mu^k(g|(\log X_0)^{k-1}(\log X_1))$.
Hence we have
$
\Gamma_{\varphi'}(t)\Gamma_{\varphi'}(-t)=
\Gamma_\varphi(t)\Gamma_g(\mu t)\Gamma_\varphi(-t)\Gamma_g(-\mu t),
$
from which we learn that 
$$
\Gamma_g(t)\Gamma_g(-t)=\lambda\frac{e^{t/2}-e^{-t/2}}{e^{\lambda t/2}-e^{-\lambda t/2}}
$$
for any $(\lambda,g)\in\mathsf{DMR}^{\B}(\mathbf k)$.
\end{rem}

\begin{rem}
Theorem \ref{thm:4:5:1606} should be related to the following set of results. In \cite{Rac}, it is proved 
that a subset of $\mathsf G^\DR(\mathbf k)$ defined by de Rham analogues of conditions (1) and (2) of 
Theorem \ref{thm:4:5:1606} is a subgroup. This result 
is used in \cite{EF0} to prove the equality of this subset with a stabilizer subgroup of $\mathsf G^\DR(\mathbf k)$
(up to degree $\leq 2$ conditions). Theorem \ref{thm:4:5:1606}, based
on \cite{EF0}, is the Betti counterpart of this result; it implies in particular the Betti analogue of 
the result of \cite{Rac}, namely that that the subset of $\mathsf G^\B(\mathbf k)$ defined by 
conditions (1) and (2) in Theorem \ref{thm:4:5:1606} is a subgroup. Whereas the proof of this result is based 
on those of \cite{EF0,Rac}, the authors do not know of a proof independent of these results. 
\end{rem}

\subsection{Equivalent definition of $\mathfrak{dmr}^\B$ }\label{sect:4:2:0307}

For $x\in\mathrm{Lie}F_2(\mathbb Q)$, set 
$$
\gamma_x(t):=\sum_{n\geq1}(-1)^{n+1}(x|(\mathrm{log}X_0)^{n-1}\mathrm{log}X_1)/n\in\mathbb Q[[t]]. 
$$

\begin{prop}
$\mathfrak{dmr}^\B=\{(\nu,x)\in\mathfrak g^\B|(x|\mathrm{log}X_0)=(x|\mathrm{log}X_1)=0$, 
$\nu=12(x|\mathrm{log}X_0\mathrm{log}X_1)$, $(x+\gamma_x(-\mathrm{log}X_1))\cdot 1_\B\in
\mathcal P(\hat{\mathcal M}^\B_{\mathbb Q})\}$, where 
$\mathcal P(\hat{\mathcal M}^\B_{\mathbb Q}):=\{m\in\hat{\mathcal M}^\B_{\mathbb Q}|
\hat\Delta^{\mathcal M,\B}(m)=m\otimes 1_\B+1_\B
\otimes m\}$.
\end{prop}

\proof 
This follows from the combination of Theorem \ref{thm:4:5:1606} and the identification of $\mathfrak{dmr}^\B$ with the 
kernel of the group morphism $\mathsf{DMR}^\B(\mathbb Q[\epsilon]/(\epsilon^2))\to\mathsf{DMR}^\B(\mathbb Q)$. 
\hfill\qed\medskip 

\section{A discrete group $\mathrm{DMR}^\B$}\label{section:discrete:group}

In \S\ref{sect:5:1:3006}, we define a discrete group $\mathrm{DMR}^\B$, which is an analogue of the discrete counterpart
$\mathrm{GT}$ of the group scheme $\mathsf{GT}(-)$, and we compute it in \S\ref{sect:5:2:3006}
(see Proposition \ref{prop:comp:GdmrB:disc}). 

\subsection{Definition of $\mathrm{DMR}^\B$}\label{sect:5:1:3006}

Recall the group inclusions 
$$
\mathsf{GT}(\mathbb Q)^{\mathrm{op}}\hookrightarrow \mathsf{DMR}^\B(\mathbb Q)\hookrightarrow \mathbb Q^\times\ltimes F_2(\mathbb Q). 
$$
The last of these groups contains the semigroup $\{\pm1\}\ltimes F_2$. 
Then $\mathrm{GT}:=\mathsf{GT}(\mathbb Q)\cap(\{\pm1\}\ltimes F_2)^{\mathrm{op}}$ is a semigroup
(see \cite{Dr}, \S4), equal to $\{\pm1\}$ (see \cite{Dr}, Proposition 4.1). 

Define similarly a semigroup 
$$
\mathrm{DMR}^\B:=\mathsf{DMR}^\B(\mathbb Q)\cap(\{\pm1\}\ltimes F_2). 
$$

\subsection{Computation of $\mathrm{DMR}^\B$}\label{sect:5:2:3006}

Let $\mathrm{ev}^{\mathcal V}:\mathcal V^\B_{\mathbb Q}\to\mathbb Q F_1$ be the algebra morphism induced by 
$X_0\mapsto 1$, $X_1\mapsto X$ ($\mathbb Q F_1$ being the algebra of the free group 
$F_1$ with one generator $X$). Since it takes $X_0-1$ to $0$, it induces a 
module morphism $\mathrm{ev}^{\mathcal M}:\mathcal M^\B_{\mathbb Q}\to\mathbb QF_1$ compatible with 
the algebra morphism $\mathrm{ev}^{\mathcal V}$. We denote by $\mathrm{ev}^{\mathcal W}:
\mathcal W^\B_{\mathbb Q}\to\mathbb Q F_1$ the restriction of $\mathrm{ev}^{\mathcal V}$ to 
$\mathcal W^\B_{\mathbb Q}$. 
Then $\mathrm{ev}^{\mathcal M}$ is also compatible with the algebra morphism $\mathrm{ev}^{\mathcal W}$.

\begin{lem}
(a) $\mathrm{ev}^{\mathcal W}$ is a Hopf algebra morphism, its source being equipped with $\Delta^{\mathcal W,\B}$
and its target with the group Hopf algebra structure. 

(b) $\mathrm{ev}^{\mathcal M}$ is a coalgebra morphism, its source being equipped with $\Delta^{\mathcal M,\B}$
and its target with the same structure as above. 
\end{lem}
 
\proof By \cite{EF1}, Propositions 2.3 and 2.4, $\mathcal W_{\mathbb Q}^\B$ is generated by the elements 
$Y_n^\pm$ ($n>0$), $X_1^{\pm1}$, where $Y_n^\pm:=(X_0^{\pm1}-1)^{n-1}X_0^{\pm1}(1-X_1^{\pm1})$; 
$\mathrm{ev}_{\mathcal W}$ is such that $X_1^{\pm1}\mapsto X^{\pm1}$, $Y_n^\pm\mapsto 0$ and the 
coproducts are such that  $X_1^{\pm1}\mapsto X_1^{\pm1}\otimes X_1^{\pm1}$, $Y_n^\pm\mapsto 
Y_n^\pm\otimes1+1\otimes Y_n^\pm+\sum_{n'+n''=n}Y_{n'}^\pm\otimes Y_{n''}^\pm$ and 
$X^{\pm1}\mapsto X^{\pm1}\otimes X^{\pm1}$; all this implies (a). (a) implies the 
commutativity of the right square in 
$$
\xymatrix{
\mathcal M^\B_{\mathbb Q}\ar_{\Delta^{\mathcal M,\B}}[d] & \ar_{(-)\cdot 1^\B}[l]\ar^{\mathrm{ev}^{\mathcal W}}[r]\mathcal W^\B_{\mathbb Q}\ar^{\Delta^{\mathcal W,\B}}[d]& \mathbb Q F_1
\ar^{\Delta_{\mathbb Q F_1}}[d]
\\ (\mathcal M^\B_{\mathbb Q})^{\otimes2}& \ar^{((-)\cdot 1^\B)^{\otimes2}}[l]\ar_{(\mathrm{ev}^{\mathcal W})^{\otimes2}}[r]
(\mathcal W^\B_{\mathbb Q})^{\otimes2}& (\mathbb Q F_1)^{\otimes2}}
$$
where $\Delta_{\mathbb Q F_1}$ is the coproduct of $\mathbb Q F_1$, while the commutativity of the left square 
follows from the definition of $\Delta^{\mathcal M,\B}$. (b) then follows from the combination of these squares, and the 
fact that $(-)\cdot 1^\B$ is a $\mathbb Q$-vector space isomorphism. \hfill\qed\medskip 

Since $\mathrm{ev}^{\mathcal V}$ is the algebra morphism underlying a group morphism $F_2\to F_1$, it
induces an algebra morphism $\mathrm{ev}^{\mathcal V,\wedge}:\hat{\mathcal V}^\B_{\mathbb Q}\to
(\mathbb Q F_1)^\wedge$. It follows that the morphisms $\mathrm{ev}^{\mathcal X}$, $\mathcal X\in
\{\mathcal W,\mathcal M\}$ induce morphisms $\mathrm{ev}^{\mathcal X,\wedge}$ between the completions of their
sources and targets. We denote by $\mathrm{iso}:(\mathbb Q F_1)^\wedge\to\mathbb Q[[t]]$ the isomorphism of 
topological Hopf algebras induced by $X\mapsto e^t$. 

\begin{lem}\label{lemma:34}
If $g\in F_2$ is such that $(\Gamma_g^{-1}(-\mathrm{log}X_1)g)\cdot 1_\B\in\mathcal G(\hat{\mathcal M}^\B_{\mathbb Q})$, 
then there exists $\lambda\in\mathbb Q$ such that $\Gamma_g(t)=e^{\lambda t}$ and $g\cdot 1^\B\in
\mathcal G(\mathcal M^\B_{\mathbb Q})$. 
\end{lem}

\proof 
There exists a unique collection $n,(a_i)_{i\in[1,n]},(b_i)_{i\in[1,n]},\alpha,\beta$, where $n\geq 0$, the $a_i,b_i$
are nonzero integers, and $\alpha,\beta$ are integers, such that 
$$
g=X_1^\alpha\cdot \prod_{i=1}^n(X_0^{a_i}X_1^{b_i})\cdot X_0^\beta,  
$$
where $\prod_{i=1}^n g_i:=g_1\cdots g_n$. One computes 
$$
g\cdot 1_\B=a\cdot 1_\B, \text{ where }a:=
\sum_{i=1}^n X_1^\alpha X_0^{a_1}X_1^{b_1}\cdots X_0^{a_i}(X_1^{b_i}-1)+X_1^\alpha\in\mathcal W^\B_{\mathbb Q}.   
$$
Then $\mathrm{ev}^{\mathcal W}(a)=\mathrm{ev}^{\mathcal V}(a)=X^{\alpha+b_1+\cdots+b_n}$.
Since $\mathrm{ev}^{\mathcal M}(g\cdot 1_\B)=\mathrm{ev}^{\mathcal W}(a)$, one obtains 
$\mathrm{iso}\circ\mathrm{ev}^{\mathcal M}(g\cdot 1_\B)=e^{(\alpha+b_1+\cdots+b_n)t}$.  
Then 
$$
\mathrm{iso}\circ\mathrm{ev}^{\mathcal M}(\Gamma_g(-\mathrm{log}X_1)g\cdot 1_\B)=
\mathrm{iso}\circ\mathrm{ev}^{\mathcal W}(\Gamma_g(-\mathrm{log}X_1))\cdot 
\mathrm{iso}\circ\mathrm{ev}^{\mathcal M}(g\cdot 1_\B)=\Gamma_g^{-1}(-t)
e^{(\alpha+b_1+\cdots+b_n)t}. 
$$
Since $\mathrm{iso}\circ\mathrm{ev}^{\mathcal M}$ takes $\mathcal G(\hat{\mathcal M}_{\mathbb Q}^\B)$
to the set $\mathcal G(\mathbb Q[[t]])$ of group-like elements of $\mathbb Q[[t]]$,  
$\Gamma_g^{-1}(-t)\cdot e^{t(\alpha+b_1+\cdots+b_n)}\in\mathcal G(\mathbb Q[[t]])$, 
and therefore $\Gamma_g^{-1}(-t)\in \mathcal G(\mathbb Q[[t]])$, so that there exists 
$\lambda\in\mathbb Q$ such that $\Gamma_g(t)=e^{\lambda t}$. 

Using the fact that $e^{\lambda\cdot\mathrm{log}X_1}\in\hat{\mathcal W}^\B_{\mathbb Q}$ 
is group-like with respect to $\hat\Delta^{\mathcal W,\B}$, the module property of
$\hat\Delta^{\mathcal M,\B}$ with respect to $\hat\Delta^{\mathcal W,\B}$ and 
$(\Gamma_g^{-1}(-\mathrm{log}X_1)g)\cdot 1_\B\in\mathcal G(\hat{\mathcal M}^\B_{\mathbb Q})$, 
we obtain $g\cdot 1_\B\in\mathcal G(\hat{\mathcal M}^\B_{\mathbb Q})$, therefore 
$g\cdot 1_\B\in\mathcal G(\mathcal M^\B_{\mathbb Q})$ as $g\cdot 1_\B\in\mathcal M^\B_{\mathbb Q}$. 
\hfill\qed\medskip 

Set for $a\in\mathbb Z$, $Y_a:=X_0^a(X_1-1)$. One derives from \cite{EF1}, Proposition 2.3, that the algebra 
$\mathcal W^\B_{\mathbb Q}$ is presented by generators $(Y_a)_{a\in\mathbb Z-\{0\}}$, $X_1^{\pm1}$ 
and relations $X_1\cdot X_1^{-1}=X_1^{-1}\cdot X_1=1$. It follows that $\mathcal W^\B_{\mathbb Q}$ 
admits an algebra grading, for which $\mathrm{deg}(Y_a)=1$ for $a\in\mathbb Z-\{0\}$ and 
$\mathrm{deg}(X_1^{\pm1})=0$. For $n\geq0$, we denote by $\mathcal W_n$ the part of 
$\mathcal W^\B_{\mathbb Q}$ of degree $n$. Then 
$$
\mathcal W^\B_{\mathbb Q}=\oplus_{n\geq0}\mathcal W_n.
$$ 
We also set $\mathcal W_{\leq n}:=\oplus_{m\leq n}\mathcal W_m$. 

\begin{lem}
One has for any $a>1$, 
\begin{equation}\label{formulas}
\Delta^{\mathcal W,\B}(Y_a)=Y_a\otimes 1+1\otimes Y_a-\sum_{a'=1}^{a-1}Y_{a'}\otimes Y_{a-a'},\quad  
\Delta^{\mathcal W,\B}(Y_{-a})=Y_{-a}\otimes X_1+X_1\otimes Y_{-a}
+\sum_{a'=1}^{a-1} Y_{-a'}\otimes Y_{a'-a}. 
\end{equation}
\end{lem}

\proof As remarked in [EF1], proof of Proposition 2.4, the formal series $s_\pm(t):=1+\sum_{k\geq 1}Y_k^\pm t^k$ are 
group-like for $\hat\Delta^{\mathcal W,\B}$. Then
$$
s_\pm(t)=1+tX_0^{\pm1}\{1-t(X_0^{\pm1}-1)\}^{-1}(1-X_1^{\pm1}) 
= (1+t-tX_0^{\pm1})^{-1}(1+t-tX_0^{\pm1}X_1^{\pm1})= \tilde s^\pm(u),
$$
where $u := t/(1 + t)$ and $\tilde s_\pm(u):=(1-uX_0^{\pm1})^{-1}(1-uX_0^{\pm1}X_1^{\pm1})$. It follows that the series 
$\tilde s_\pm(u)$ are group-like. The expansions $\tilde s_\pm(u)
=1+\sum_{k\geq1}u^kX_0^{\pm k}(1-X_1^{\pm1})$ imply $\tilde s_+(u) =1-\sum_{k\geq1}u^k Y_k$, which implies 
the first part of \eqref{formulas}, and $\tilde s_-(u) =1+\sum_{k\geq1}u^k Y_{-k}X_1^{-1}$, which together with the 
group-likeness of $X_1$, implies its second part. 
\hfill\qed\medskip 

\begin{rem}
One can show that \eqref{formulas} remains valid for any $a\in\mathbb Z$ 
under the summation convention of \cite{EF1}, (2.4.9). 
\end{rem}

\begin{lem}
The coproduct $\Delta^{\mathcal W,\B}:\mathcal W^\B_{\mathbb Q}\to(\mathcal W^\B_{\mathbb Q})^{\otimes2}$ 
is such that $\Delta^{\mathcal W,\B}(\mathcal W_{\leq n})\subset (\mathcal W_{\leq n})^{\otimes2}$. 
\end{lem}

\proof This follows that the equalities $\Delta^{\mathcal W,\B}(X_1^{\pm1})=X_1^{\pm1}\otimes X_1^{\pm1}$ and 
\eqref{formulas}
for any $a>0$ (see \cite{EF1}, Proposition 2.4). \hfill\qed\medskip 

\begin{lem}
There is an algebra morphism $\Delta^{\mathcal W,\B}_{\mathrm{mod}}:\mathcal W^\B_{\mathbb Q}\to
(\mathcal W^\B_{\mathbb Q})^{\otimes2}$, defined by $X_1^{\pm1}\mapsto X_1^{\pm1}\otimes X_1^{\pm1}$ 
and $Y_{\pm a}\mapsto\mp\sum_{a'=1}^{a-1}Y_{\pm a'}\otimes Y_{\pm(a-a')}$ for $a>0$. Then 
$\Delta^{\mathcal W,\B}_{\mathrm{mod}}(\mathcal W_n)\subset \mathcal W_n^{\otimes2}$ and the diagram 
\begin{equation}\label{diag:Hopf:n}
\xymatrix{
\mathcal W_{\leq n} \ar^{\Delta^{\mathcal W,\B}}[r]\ar_{\mathrm{pr}_n}[d]& (\mathcal W_{\leq n})^{\otimes2}
\ar^{\mathrm{pr}_n^{\otimes2}}[d]\\ 
\mathcal W_n \ar_{\Delta^{\mathcal W,\B}_{\mathrm{mod}}}[r]& \mathcal W_n^{\otimes2} }
\end{equation}
commutes, where $\mathrm{pr}_n:\mathcal W_{\leq n}\to\mathcal W_n$ is the projection on the highest degree component. 
\end{lem}

\proof Immediate. \hfill\qed\medskip 

For $a_1,\ldots,a_n\in\mathbb Z-\{0\}$, set
$$
\mathcal W(a_1,\ldots,a_n):=\mathrm{Span}_{\mathbb Q}\{X_1^{b_0}Y_{a_1}X_1^{b_1}\cdots
X_1^{b_{n-1}}Y_{a_n}X_1^{b_n}|b_0,\ldots,b_n\in\mathbb Z\}\subset\mathcal W_n.
$$ 
\begin{lem}
$$
\mathcal W_n=\oplus_{a_1,\ldots,a_n\in\mathbb Z-\{0\}}\mathcal W(a_1,\ldots,a_n)
$$
\end{lem}

\proof 
This follows from the presentation of $\mathcal W^\B_{\mathbb Q}$. 
\hfill\qed\medskip 

For $a\in\mathbb Z-\{0\}$, set $S(a):=\{(a',a'')\in\mathbb Z-\{0\}\mid \mathrm{sgn}(a')=\mathrm{sgn}(a'')=\mathrm{sgn}(a)$ and 
$a'+a''=a\}$. 

\begin{lem}\label{lemma:38}
For $a_1,\ldots,a_n\in\mathbb Z-\{0\}$, one has 
$$
\Delta^{\mathcal W,\B}_{\mathrm{mod}}(\mathcal W(a_1,\ldots,a_n))\subset
\oplus_{(a'_1,a''_1)\in S(a_1),\ldots,(a'_n,a''_n)\in S(a_n)}
\mathcal W(a'_1,\ldots,a'_n)\otimes \mathcal W(a''_1,\ldots,a''_n).
$$
\end{lem}

\proof This follows from (\ref{formulas}). \hfill\qed\medskip 

\begin{lem}\label{lemma:form:of:grouplike}
Let $g\in F_2$. Then $g\cdot 1_\B\in\mathcal G(\mathcal M^\B_{\mathbb Q})$ 
iff there exist $\alpha,\beta\in\mathbb Z$, such that $g=X_1^\alpha X_0^\beta$. 
\end{lem}

\proof The group $\mathbb Z^2$ acts on the set $F_2$ by $(\alpha,\beta)\bullet g:=X_1^\alpha g X_0^\beta$. 
For $n\geq0$, set 
$$
(F_2)_n:=\{X_0^{a_1}X_1^{b_1}\cdots X_0^{a_n}X_1^{b_n}\mid
a_1,\ldots,b_n\in\mathbb Z-\{0\}\}\subset F_2. 
$$
Then the composition 
$$
\sqcup_{n\geq0}(F_2)_n\to F_2\to F_2/\mathbb Z^2
$$
is a bijection. 

Set $S:=\{g\in F_2\mid g\cdot 1_\B\in\mathcal G(\mathcal M^\B_{\mathbb Q})\}$. One checks that $S$ is stable 
under the action of $\mathbb Z^2$. It follows that 
\begin{equation}\label{equation:S}
S=\sqcup_{n\geq0}\mathbb Z^2\bullet((F_2)_n\cap S).
\end{equation} 

One has 
\begin{equation}\label{equation:S:zero}
(F_2)_0\cap S=\{e\}. 
\end{equation} 
We now compute $(F_2)_n\cap S$ for $n>0$. 
Let $g\in(F_2)_n$. Let $a_1,b_1,\ldots,a_n,b_n\in\mathbb Z-\{0\}$ be such that 
$g=X_0^{a_1}X_1^{b_1}\cdots X_0^{a_n}X_1^{b_n}$ and set
\begin{equation}\label{equation:w(g)}
w(g):=1+\sum_{i=1}^n X_0^{a_1}X_1^{b_1}\cdots X_0^{a_i}(X_1^{b_i}-1)\in\mathcal W^\B_{\mathbb Q}. 
\end{equation} 
The summand in the right-hand side of \eqref{equation:w(g)} corresponding to index $i$ belongs to $\mathcal W_{\leq i}$ 
and $1\in\mathcal W_0$. It follows that 
$w(g) \in\mathcal W_{\leq n}$ and that 
$$
w(g) \equiv X_0^{a_1}X_1^{b_1}\cdots X_0^{a_n}(X_1^{b_n}-1)
\text{ mod }\mathcal W_{\leq n-1}. 
$$
Moreover the elements $X_0^{a_1}X_1^{b_1}\cdots X_0^{a_n}(X_1^{b_n}-1)$ and $X_0^{a_1}(X_1^{b_1}-1)\cdots 
X_0^{a_n}(X_1^{b_n}-1)$ of $\mathcal W_{\leq n}$ are equivalent mod $\mathcal W_{\leq n-1}$. All this implies that 
\begin{equation}\label{eqn}
w(g)\equiv Y_{a_1}\cdot \varphi_{b_1}(X_1)\cdots Y_{a_n}\cdot\varphi_{b_n}(X_1)
\text{ mod }\mathcal W_{\leq n-1},  
\end{equation} 
where for $b\in\mathbb Z-\{0\}$, we set $\varphi_b(t):=(t^b-1)/(t-1)\in\mathbb Z[t,t^{-1}]$. 
(\ref{eqn}) and the fact that for $b\neq0$, $\varphi_b\neq0$ implies that 
\begin{equation}\label{eq:nonzero}
\mathrm{pr}_n(w(g))\in\mathcal W(a_1,\ldots,a_n)-\{0\}.
\end{equation} 

Assume now that $g\in(F_2)_n\cap S$. One has $g\cdot 1_\B=w(g)\cdot 1_\B$, therefore $g\cdot 1_\B
\in\mathcal G(\mathcal M^\B_{\mathbb Q})$ is equivalent to the group-likeness of $w(g)\in\mathcal W^\B_{\mathbb Q}$ 
for $\Delta^{\mathcal W,\B}$. Since $w(g)\in\mathcal W_{\leq n}$, the diagram (\ref{diag:Hopf:n}) implies
\begin{equation}\label{strat:equation}
\Delta^{\mathcal M,\B}_{\mathrm{mod}}(\mathrm{pr}_n(w(g)))
=\mathrm{pr}_n(w(g))^{\otimes2}
\end{equation} 
(equality in $\mathcal W_n^{\otimes2}$).
By Lemma \ref{lemma:38}, the left-hand side of (\ref{strat:equation}) belongs to 
$$
\oplus_{(a'_1,a''_1)\in S(a_1),\ldots,(a'_n,a''_n)\in S(a_n)}\mathcal W(a'_1,\ldots,a'_n)\otimes \mathcal W(a''_1,\ldots,a''_n)
\subset (\mathcal W_n)^{\otimes2}, 
$$
while the right-hand side belongs to 
$$
\mathcal W(a_1,\ldots,a_n)\otimes \mathcal W(a_1,\ldots,a_n)
\subset (\mathcal W_n)^{\otimes2}.  
$$
By the direct sum decomposition 
$$
\mathcal W_n^{\otimes2}=\oplus_{((a_1,b_1),\ldots,(a_n,b_n))\in((\mathbb Z-\{0\})^2)^n}
\mathcal W(a_1,\ldots,a_n)\otimes \mathcal W(b_1,\ldots,b_n)
$$
and since $((a_1,a_1),\ldots,(a_n,a_n))\notin S(a_1)\times\cdots\times S(a_n)$ (as $(a,a)\notin S(a)$ for any $a\neq0$), 
both sides of (\ref{strat:equation}) should be zero, which contradicts (\ref{eq:nonzero}). All this implies that 
$(F_2)_n\cap S=\emptyset$ for $n>0$. Together with (\ref{equation:S}) and (\ref{equation:S:zero}), this implies 
Lemma \ref{lemma:form:of:grouplike}. \hfill\qed\medskip  

\begin{prop}\label{prop:comp:GdmrB:disc} 
There is an isomorphism $\mathrm{DMR}^\B\simeq\{\pm1\}$. 
\end{prop}

\proof One obviously has $\{(\pm1,1)\}\subset \mathrm{DMR}^\B$. Let us prove the opposite inclusion. 
Let 
$$
(\mu,g)\in \mathsf{DMR}^\B(\mathbb Q)\cap (\{\pm1\}^\times \times F_2). 
$$
By Lemmas \ref{lemma:34} and \ref{lemma:form:of:grouplike}, the condition that 
$\Gamma_g^{-1}(-\mathrm{log}X_1)g\cdot 1_\B\in\mathcal G(\hat{\mathcal M}^\B_{\mathbb Q})$
implies that for some $\alpha,\beta\in\mathbb Z$, one has 
$g=X_1^\alpha X_0^\beta$. The conditions $(g|\mathrm{log}X_0)=(g|\mathrm{log}X_1)=0$
then imply $\alpha=\beta=0$, therefore $g=1$.  This proves Proposition \ref{prop:comp:GdmrB:disc}. 
\hfill\qed\medskip 

\begin{rem}
Using the proof of Proposition \ref{prop:comp:GdmrB:disc}, one can prove the stronger result 
$\mathsf{DMR}^\B(\mathbb Q)\cap (\mathbb Q^\times \times F_2)=\{\pm1\}$. 
Indeed, this proof implies that if $(\mu,g)$ belongs to this intersection, then 
$g=1$. The condition $\mu^2=1+24(g|\mathrm{log}X_0\mathrm{log}X_1)$
then implies that $\mu=\pm1$. 
\end{rem}

\begin{rem}
Proposition \ref{prop:comp:GdmrB:disc} is consistent with the conjectural equality of Lie algebras 
$\mathfrak{grt}_1=\mathfrak{dmr}_0$. Indeed, this equality is equivalent to $\mathsf{DMR}^\DR(-)=\mathsf{GRT}(-)$, 
which via the isomorphism $i_{(1,\Phi)}$, $\Phi\in\mathsf M_1(\mathbb Q)$  is 
equivalent to $\mathsf{DMR}^\B(-)=\mathsf{GT}(-)$, which upon taking rational points and intersecting with $\{\pm1\}\times F_2$
implies the equality $\mathrm{DMR}^\B=\mathrm{GT}$, which is Proposition \ref{prop:comp:GdmrB:disc}. 
\end{rem}

\section{Pro-$p$ aspects}\label{section:pro-l}

In this section, we first recall some material on the relation between the pro-$p$ and prounipotent completions of 
discrete groups (\S\ref{section:plapucodg}), $p$ being a prime number. In \S\ref{subsection:rugtl}, we recall the definition of the pro-$p$ 
analogue $\mathrm{GT}_p$ of the  Grothendieck-Teichm\"uller group, and we use the results of 
\S\ref{section:plapucodg} to prove a statement of \cite{Dr} on the relations of $\mathrm{GT}_p$ 
with $\mathsf{GT}(\mathbb Q_p)$ (Corollary \ref{cor:gtlsubsetgtql}); we also make precise the relation between 
$\mathrm{GT}_p$ and the semigroup $\underline{\mathrm{GT}}_p$ introduced in \cite{Dr} (Proposition \ref{cor:1:12}). 
We then define a group $\mathrm{DMR}_p^\B$ (see Definition \ref{defin:GpdmrB}) and show that it fits in a commutative diagram, 
which makes it into a natural pro-$p$ analogue of the group scheme $\mathsf{DMR}^\B(-)$ (\S\ref{section:final}). 

\subsection{Pro-$p$ and prounipotent completions of discrete groups}\label{section:plapucodg}

\subsubsection{A morphism $\Gamma^{(p)}\to\Gamma(\mathbb Q_p)$}

If $\Gamma$ is a group, we denote by $\Gamma^{(p)}$ its pro-$p$ completion. If $\mathbf k$ is a $\mathbb Q$-algebra, 
we denote by $\Gamma(\mathbf k)$ the group of $\mathbf k$-points of its prounipotent completion. We also denote by 
$\mathrm{Lie}(\Gamma)$ the Lie algebra of this prounipotent completion.  

\begin{lem}[\cite{HM}, Lemma A.7]
Suppose that $\Gamma$ is finitely generated discrete group, then there is a continuous homomorphism 
$\Gamma^{(p)}\to\Gamma(\mathbb Q_p)$ compatible with the morphisms from $\Gamma$ to its source and target. 
\end{lem}

When $\Gamma$ is the free group $F_n$, this gives
a continuous homomorphism $F_n^{(p)}\to F_n(\mathbb Q_p)$. 

\subsubsection{Injectivity of $F_n^{(p)}\to F_n(\mathbb Q_p)$}

Let $\Gamma$ be a group. Define the $\mathbb Z_p$-algebra of $\Gamma$, denoted 
$\mathbb Z_p[\![\Gamma]\!]$, to be the inverse limit of the group algebras of 
the quotients of $\Gamma$ which are $p$-groups with coefficients in 
$\mathbb Z_p$. This is a topological Hopf algebra.
(When $\Gamma$ is a pro-$p$ group,  $\mathbb Z_p[\![\Gamma]\!]$ coincides with the object introduced in 
\cite{Serre}, p. ~7.) 

If $H$ is a (topological) Hopf algebra, we denote by 
$\mathcal G(H)$ the group of its group-like elements.

\begin{lem}\label{lemma:gplikes:l-adic}
The group $\mathcal G(\mathbb Z_p[\![\Gamma]\!])$ is equal to $\Gamma^{(p)}$. 
\end{lem}

\proof The group $\mathcal G(\mathbb Z_p[\![\Gamma]\!])$ is 
the inverse limit of the groups of group-like elements 
of the group algebras $\mathbb Z_pK$, where 
$K$ runs over all the quotients of $\Gamma$ which are $p$-groups. As 
$\mathcal G(\mathbb Z_pK)$ is equal to $K$, 
$\mathcal G(\mathbb Z_p[\![\Gamma]\!])$ is equal to the inverse 
limit of the finite quotients of $\Gamma$ which are $p$-groups, therefore to $\Gamma^{(p)}$. 
\hfill\qed\medskip 
 
\begin{lem}\label{lemma:F_n^l}
Let $A(n) :=\mathbb Z_p\langle\langle
t_1,\ldots,t_n\rangle\rangle$ 
be the algebra of associative formal power series 
in variables $t_1,\ldots,t_n$ with coefficients in 
$\mathbb Z_p$, equipped with the topology of convergence 
of coefficients. Then $A(n)$ has a Hopf algebra structure
with coproduct given by $t_i\mapsto t_i\otimes1+1\otimes t_i+t_i\otimes t_i$
for $i=1,\ldots,n$. Let $F_n$ be the free group over 
generators $X_1,\ldots,X_n$. There is an isomorphism 
$$
F_n^{(p)}\simeq\mathcal G(A(n))
$$ 
induced by $X_i\mapsto 1+t_i$ for $i=1,\ldots,n$. 
\end{lem}

\proof This follows from Lemma \ref{lemma:gplikes:l-adic}
combined with the isomorphism $\mathbb Z_p[\![F_n^{(p)}]\!]
\simeq A(n)$, see \cite{Serre},  \S I.1.5. Proposition 7. 
\hfill\qed\medskip  

\begin{lem}\label{lemma:injectivity:free}
The map $F_n^{(p)}\to F_n(\mathbb Q_p)$ is injective. 
\end{lem}  
 
\proof If $\mathbf k$ is a $\mathbb Q$-algebra, there is 
an isomorphism $(\mathbf kF_n)^\wedge
\simeq \mathbf k\langle\langle u_1,\ldots,u_n\rangle\rangle$, where each $u_i$ is primitive. 
Moreover, $F_n(\mathbf k)=\mathcal G((\mathbf kF_n)^\wedge)$, 
therefore 
$$
F_n(\mathbf k)=\mathcal G(\mathbf k\langle\langle 
u_1,\ldots,u_n\rangle\rangle). 
$$
The result now follows from the specialization of this result 
for $\mathbf k=\mathbb Q_p$, from the topological Hopf algebra 
inclusion $\mathbb Z_p\langle\langle
t_1,\ldots,t_n\rangle\rangle \subset \mathbb Q_p\langle\langle 
u_1,\ldots,u_n\rangle\rangle$ given by $t_i\mapsto e^{u_i}-1$, and from Lemma 
\ref{lemma:F_n^l}. \hfill\qed\medskip 

\subsubsection{Exact sequences of pro-$p$ completions}

\begin{lem}[\cite{Ih:G}, see also \cite{And}] \label{lemma:Ihara}
Let $1\to N\to G\to H\to 1$ be an
exact sequence of discrete groups such that (i) 
$(G, N) = (N, N)$, and (ii)
$N$ is a free group of finite rank greater than $1$. Then the induced sequence of pro-$p$
completions $1\to N^{(p)}\to G^{(p)}\to H^{(p)}\to 1$
is exact. 
\end{lem}

For $n\geq3$, let $K_n$ be the Artin pure braid group with $n$ strands. It is presented by generators 
$x_{ij}$, $1\leq i<j\leq n$ and relations 
$$
(a_{ijk},x_{ij})=(a_{ijk},x_{ik})=(a_{ijk},x_{jk})=1, \quad i<j<k, \quad a_{ijk}=x_{ij}x_{ik}x_{jk}, 
$$
$$
(x_{ij},x_{kl})=
(x_{il},x_{jk})=1, \quad (x_{ik},x_{ij}^{-1}x_{jl}x_{ij})\quad if\quad i<j<k<l.  
$$
For any $i\in[\![1,n]\!]$, the elements $x_{i1},\ldots,x_{in}$ of $K_n$ generate a subgroup 
isomorphic to $F_{n-1}$, and we have an exact sequence 
\begin{equation}\label{exact:K}
1\to F_{n-1}\to K_n\to K_{n-1}\to 1.  
\end{equation}

\begin{lem}
If $n\geq 3$, then the exact sequence \eqref{exact:K} induces a short exact sequence of pro-$p$ groups:
$$
1\to F_{n-1}^{(p)}\to K_n^{(p)}\to K_{n-1}^{(p)}\to 1. 
$$
\end{lem}

\begin{proof} Let $P_{n+1}$ be the pure sphere group of the sphere with $n+1$ marked points (cf. \cite{EF1}).
It is known that $P_{n+1}$ is isomorphic to the quotient $K_n/Z(K_n)^2$, where $Z(K_n)$
is the center of $K_n$. 

In \cite{Ih:G}, Proposition 2.3.1, it is shown that for any $j\in[\![1,n]\!]\setminus\{i\}$, 
$P_{n+1}$ is equal to the product 
$\langle \bar x_{1i},\ldots,\bar x_{in}\rangle\cdot C(\bar x_{ij})$, where 
the projection map $K_n\to P_{n+1}$ is denoted $g\mapsto \bar g$, and where 
$C(\bar x_{ij})$ is the centralizer subgroup of $\bar x_{ij}$.

Since $Z(K_n)^2$ is contained in $C(x_{ij})$, $K_n$ is equal to the product $\langle x_{1i},\ldots,x_{in}\rangle
\cdot C(x_{ij})$, where $C(x_{ij})$ is the centralizer subgroup of $x_{ij}$.

Then any $k\in K_n$ can be expressed as $f\cdot c$, where $f\in \langle x_{1i},\ldots,x_{in}\rangle$ and $c\in C(x_{ij})$. 
Then $(k,x_{ij})=(f\cdot c,x_{ij})=(f,x_{ij})\in (F_{n-1},F_{n-1})$. As this holds for any $j\in[\![1,n]\!]\setminus\{i\}$, 
one obtains $(K_n,F_{n-1})\subset (F_{n-1},F_{n-1})$, therefore the equality of these subgroups of $K_n$ as the opposite inclusion 
is obvious. 

One can therefore apply Lemma \ref{lemma:Ihara} to the exact sequence (\ref{exact:K}), which yields the result.  
\end{proof}

\subsubsection{Exact sequences of prounipotent completions}

\begin{lem}
Let $\mathbf k$ be a $\mathbb Q$-algebra and let $n\geq3$.
The exact sequence \eqref{exact:K} induces a short exact sequence of groups:
$$
1\to F_{n-1}(\mathbf k)\to K_n(\mathbf k)\to K_{n-1}(\mathbf k)\to 1. 
$$
\end{lem}

\begin{proof} 
According to \cite{Dr}, any associator $\Phi\in\mathsf M_1(\mathbb Q)$ and  
parenthesization $P$ of a word with $n$ identical letters gives rise to an isomorphism 
$b^P_\Phi:\mathrm{Lie}(K_n)\to\hat{\mathfrak t}_n$, where $\mathfrak t_n$ is the graded $\mathbb Q$-Lie 
algebra with degree one generators $t_{ij}$, $i\neq j\in[\![1,n]\!]$ and relations $t_{ji}=t_{ij}$, 
$[t_{ik}+t_{jk},t_{ij}]=0$, and $[t_{ij},t_{kl}]=0$ for $i,j,k,l$ all distinct, and where $\hat{\mathfrak t}_n$
is its degree completion.

The morphisms of (\ref{exact:K}) induce Lie algebra morphisms $\mathrm{Lie}(F_{n-1})\to \mathrm{Lie}(K_n)$ and 
$\mathrm{Lie}(K_n)\to \mathrm{Lie}(K_{n-1})$. One has a commutative diagram 
$$
\xymatrix{\mathrm{Lie}(K_n) \ar[d]\ar^{b^P_\Phi}[r]& \hat{\mathfrak t}_n\ar[d]\\ 
\mathrm{Lie}(K_{n-1}) \ar^{b^{P^i}_\Phi}[r]& \hat{\mathfrak t}_{n-1} }
$$
where $P^i$ is $P$ with the $i$-th letter erased, and where the right vertical arrow is induced by the 
morphism $\mathfrak t_n\to\mathfrak t_{n-1}$, $t_{ia}\mapsto 0$, $t_{ab}\mapsto t_{f(a)f(b)}$
for $a,b\in[\![1,n]\!]-\{i\}$, $f$ being the increasing bijection $[\![1,n]\!]-\{i\}\simeq[\![1,n-1]\!]$. 

It follows from this diagram that $b_\Phi^P$ restricts to a Lie algebra morphism $\mathrm{Lie}(F_{n-1})\to\hat{\mathfrak f}_{n-1}$, 
where $\hat{\mathfrak f}_{n-1}$ is the kernel of $\hat{\mathfrak t}_n\to\hat{\mathfrak t}_{n-1}$, which is the 
degree completion of the kernel of $\mathfrak t_n\to\mathfrak t_{n-1}$, a Lie algebra freely generated by the 
$t_{ij}$, $j\in[\![1,n]\!]-\{i\}$. The abelianization of this morphism can be shown to be an isomorphism, therefore 
$\mathrm{Lie}(F_{n-1})\to\hat{\mathfrak f}_{n-1}$ is an isomorphism. 

In the following diagram
$$
\xymatrix{
0\ar[r] & \hat{\mathfrak f}_{n-1}\ar[r]\ar[d] & \hat{\mathfrak t}_n\ar[r]\ar[d] & 
\hat{\mathfrak t}_{n-1}\ar[r]\ar[d] & 0\\
0\ar[r] & \mathrm{Lie}F_{n-1}\ar[r] & \mathrm{Lie}K_n\ar[r] & 
\mathrm{Lie}K_{n-1}\ar[r] & 0}
$$
the top sequence is exact. Since the vertical arrows are isomorphisms 
and since the squares commute, it follows that the bottom sequence is exact. The result follows. 
\end{proof}

\subsubsection{Injectivity of $K_n^{(p)}\to K_n(\mathbb Q_p)$}

\begin{lem}
For any $n\geq2$, the map $K_n^{(p)}\to K_n(\mathbb Q_p)$ is injective.
\end{lem}

\begin{proof}
The statement is obvious for $n=2$ as $K_2\simeq\mathbb Z$. 
One then proceeds by induction over $n$. Assume that the statement holds for $n-1$, then 
we have a natural morphism between two exact sequences, 
which makes the following diagram commutative 
$$
\xymatrix{
1\ar[r] & F_{n-1}^{(p)} \ar[d]\ar[r] & K_n^{(p)}\ar[d]\ar[r] & K_{n-1}^{(p)}\ar[d]\ar[r] & 1 \\
1\ar[r] & F_{n-1}(\mathbb Q_p) \ar[r] & K_n(\mathbb Q_p)\ar[r] & K_{n-1}(\mathbb Q_p)\ar[r] & 1 .
}
$$
The leftmost vertical map is injective by Lemma \ref{lemma:injectivity:free} and the rightmost vertical map is as well 
by the induction assumption, which shows that the middle vertical map is injective. 
\end{proof}

\subsection{Results on $\mathrm{GT}_p$}\label{subsection:rugtl}

Let $F_2$ be the free group with generators $X_0,X_1$ (see \S\ref{sect:2:1:1:2606}). 

A semigroup structure is defined on $\mathbb Z_p\times F_2^{(p)}$ by $(\lambda_1,f_1)
\circledast(\lambda_2,f_2)=(\lambda_1\lambda_2,f)$, where 
$$
f(X_0,X_1)=f_2(f_1(X_0,X_1)X_0^{\lambda_1}f_1(X_0,X_1)^{-1},X_1^{\lambda_1})f_1(X_0,X_1)
$$
(this is the opposite of the product of \cite{Dr}, (4.11), with $X_0,X_1$ replacing $X,Y$). 

\begin{lem}\label{prop:invertible}
$(\mathbb Z_p^\times\times F_2^{(p)},\circledast)$ is a subgroup of $(\mathsf G^\B(\mathbb Q_p),\circledast)$
(see Lemma \ref{lem:2:1:2606}). 
\end{lem}

\proof Recall the subsets $F_2(\mathbb Q_p)=\mathcal G(\mathbb Q_p\langle\langle t_0,t_1\rangle\rangle)$ 
and $1+\mathbb Z_p\langle\langle t_0,t_1\rangle\rangle_0$ of $\mathbb Q_p\langle\langle t_0,t_1\rangle\rangle^\times$. 
By Lemma \ref{lemma:F_n^l}, one has 
$$
F_2^{(p)}=\mathcal G(\mathbb Q_p\langle\langle t_0,t_1\rangle\rangle)\cap(1+\mathbb Z_p\langle\langle t_0,t_1\rangle\rangle_0) 
\subset \mathbb Q_p\langle\langle t_0,t_1\rangle\rangle^\times,  
$$
therefore
$$
\mathbb Z_p^\times\times F_2^{(p)}=(\mathbb Q_p^\times\times\mathcal G(\mathbb Q_p\langle\langle t_0,t_1\rangle\rangle))
\cap(\mathbb Z_p^\times\times(1+\mathbb Z_p\langle\langle t_0,t_1\rangle\rangle_0)) 
\subset \mathbb Q_p^\times\times\mathbb Q_p\langle\langle t_0,t_1\rangle\rangle^\times,  
$$ 
Then $\mathbb Q_p^\times\times\mathbb Q_p\langle\langle t_0,t_1\rangle\rangle^\times$ is equipped with the group structure 
$$
(\lambda,a)\circledast(\mu,b):=(\lambda\mu,a(t_0,t_1)b((1+t_0)^\lambda-1,a^{-1}((1+t_1)^\lambda-1)a)),  
$$
and $\mathbb Q_p^\times\times\mathcal G(\mathbb Q_p\langle\langle t_0,t_1\rangle\rangle$ is then a subgroup, 
which identifies with $\mathsf G^\B(\mathbb Q_p)$ under the identifications $X_i=1+t_i$, $i=0,1$. 

$\mathbb Z_p^\times\times(1+\mathbb Z_p\langle\langle t_0,t_1\rangle\rangle_0)$ is a sub-semigroup of 
$(\mathbb Q_p^\times\times\mathbb Q_p\langle\langle t_0,t_1\rangle\rangle^\times,\circledast)$. 
Let $(\mu,b)\in \mathbb Z_p^\times\times(1+\mathbb Z_p\langle\langle t_0,t_1\rangle\rangle_0)$ and let 
$(1/\mu,a)\in
\mathbb Q_p^\times\times\mathbb Q_p\langle\langle t_0,t_1\rangle\rangle^\times$ be its inverse. One has then 
\begin{equation}\label{eq:a:b}
a(t_0,t_1)b((1+t_0)^{1/\lambda}-1,a^{-1}[(1+t_1)^{1/\lambda}-1]a)=1. 
\end{equation}
One shows by induction on $n$ that 
\begin{equation}\label{approx:a}
a\in 1+\mathbb Z_p\langle\langle t_0,t_1\rangle\rangle_0+\mathbb Q_p\langle\langle t_0,t_1\rangle\rangle_{>n}. 
\end{equation}
For $n=0$, this follows from (\ref{eq:a:b}). Assume that (\ref{approx:a}) holds for $n$. Then 
$$
b((1+t_0)^{1/\lambda}-1,a^{-1}[(1+t_1)^{1/\lambda}-1]a)
\in1+\mathbb Z_p\langle\langle t_0,t_1\rangle\rangle_0+\mathbb Q_p\langle\langle t_0,t_1\rangle
\rangle_{>n+1}.
$$ 
(\ref{eq:a:b}) then implies (\ref{approx:a}) with $n$ replaced by $n+1$. Finally $a\in1+\mathbb Z_p\langle\langle t_0,t_1\rangle\rangle_0$. 
It follows that
$\mathbb Z_p^\times\times(1+\mathbb Z_p\langle\langle t_0,t_1\rangle\rangle_0)$ is a subgroup of the group
$(\mathbb Q_p^\times\times\mathbb Q_p\langle\langle t_0,t_1\rangle\rangle^\times,\circledast)$. 
It follows that $\mathbb Z_p^\times\times F_2^{(p)}$ is the intersection of two subgroups of this group, 
which implies the result. 
\hfill\qed\medskip 

\begin{lem}\label{invertibles:in:sd:product}
The subgroup of invertible elements of $(\mathbb Z_p\times F_2^{(p)},\circledast)$ is $\mathbb Z_p^\times\times F_2^{(p)}$. 
\end{lem}

\proof This follows from Lemma \ref{prop:invertible}.  
\hfill\qed\medskip 

Consider the morphisms from $F_2$ to various groups given by the 
following table: 

\begin{center}
\begin{tabular}{|c|c|c|c|c|c|c|c|c|c|c|}
  \hline
name of morphism & $\theta$ & $\kappa$ & $\alpha_1$& $\alpha_2$& $\alpha_3$& $\alpha_4$& $\alpha_5$  \\
  \hline
 target group & $F_2$ & $F_2$ & $K_4$ & $K_4$ & $K_4$ & $K_4$ & $K_4$   \\ 
\hline   
image of $X_0$  & $X_1$ & $X_1$ & $x_{23}x_{24}$ &  $x_{12}$& $x_{23}$   &  $x_{34}$ & $x_{12}x_{13}$   \\
  \hline image of $X_1$ & $X_0$ & $(X_0X_1)^{-1}$  & $x_{12}$ & $x_{23}$ & $x_{34}$ & $x_{13}x_{23}$ &  $x_{24}x_{34}$  \\ \hline
\end{tabular}
\end{center}
The pro-$p$ completions of these morphisms are denoted in the same way. 

In \cite{Dr}, p. 846, $\underline{\mathrm{GT}}_p$ is defined as the set of all $(\lambda,f)\in(1+2\mathbb Z_p)\times F_2^{(p)}$ 
such that 
\begin{equation}\label{defining:conditions}
f\theta(f)=1, \quad \kappa^2(f)(X_0X_1)^{-m} \kappa(f) X_1^m f X_0^m=1, \quad\alpha_1(f)\alpha_3(f)\alpha_5(f)\alpha_2(f)\alpha_4(f)=1, 
\end{equation}
where $m=(\lambda-1)/2$ (equalities in $F_2^{(p)}$ and $K_4^{(p)}$) and $\mathsf{GT}(\mathbb Q_p)$ is the set of all 
$(\lambda,f)\in\mathbb Q_p^\times\times F_2(\mathbb Q_p)$ such that the same identities hold in $F_2(\mathbb Q_p)$ 
and $K_4(\mathbb Q_p)$. 

The subset $\underline{\mathrm{GT}}_p\subset\mathbb Z_p\times F_2^{(p)}$ is shown to be a sub-semigroup; it is equipped 
with the structure opposite to that induced by  $\mathbb Z_p\times F_2^{(p)}$. The group  ${\mathrm{GT}}_p\subset \underline{\mathrm{GT}}_p$ is then defined to be the group 
of invertible elements in $\underline{\mathrm{GT}}_p$.

\begin{cor}\label{cor:13/07/2018}
${\mathrm{GT}}_p=\underline{\mathrm{GT}}_p\cap(\mathbb Z_p^\times\times F_2^{(p)})$. 
\end{cor}

\proof This follows from Lemma \ref{invertibles:in:sd:product} and the definition of ${\mathrm{GT}}_p$.  
\hfill\qed\medskip

\begin{prop}\label{prop:coup:de:fil}
${\mathrm{GT}}_p={\mathsf{GT}}(\mathbb Q_p)\cap(\mathbb Z_p^\times\times F_2^{(p)})$. 
\end{prop}

\proof ${\mathsf{GT}}(\mathbb Q_p)$ is the subset of $\mathbb Q_p^\times\times F_2(\mathbb Q_p)$
defined by the prounipotent versions of the conditions (\ref{defining:conditions}), while ${\mathrm{GT}}_p$
is the subset of $((1+2\mathbb Z_p)\cap\mathbb Z_p^\times)\times F_2^{(p)}=\mathbb Z_p^\times\times F_2^{(p)}$ 
defined by the pro-$p$ versions of the same conditions (the equality follows from $1+2\mathbb Z_p=\mathbb Z_p^\times$ for $p=2$, and 
$1+2\mathbb Z_p=\mathbb Z_p$ for $p\neq2$).  
Moreover, the inclusion $K_4^{(p)}\subset K_4(\mathbb Q_p)$ and the compatibilities of the pro-$p$ and prounipotent
completions of a given group morphism imply that an element of $\mathbb Z_p^\times\times F_2^{(p)}$ satisfies the 
pro-$p$ version of (\ref{defining:conditions}) iff its image in $\mathbb Q_p^\times\times F_2(\mathbb Q_p)$
satisfies its prounipotent version. 
\hfill\qed\medskip 

\begin{rem}
Proposition \ref{prop:coup:de:fil} shows that ${\mathrm{GT}}_p$ is the intersection of two subgroups of 
$\mathsf G^\B(\mathbb Q_p)^{\mathrm{op}}$, 
and is therefore a group. 
\end{rem}

\begin{cor} [\cite{Dr}, p. 846]\label{cor:gtlsubsetgtql}
${\mathrm{GT}}_p\subset {\mathsf{GT}}(\mathbb Q_p)$. 
\end{cor}

\proof This immediately follows from Proposition \ref{prop:coup:de:fil}. 
\hfill\qed\medskip 

\begin{prop}\label{cor:1:12}
$\mathrm{GT}_p=\underline{\mathrm{GT}}_p$ for $p=2$ and $\mathrm{GT}_p=\underline{\mathrm{GT}}_p\times_{\mathbb Z_p}
\mathbb Z_p^\times$ for $p>2$. 
\end{prop}

\proof This follows from Corollary \ref{cor:13/07/2018} together with $1+2\mathbb Z_p=\mathbb Z_p^\times$ for $p=2$, and 
$1+2\mathbb Z_p=\mathbb Z_p$ for $p\neq2$. 
\hfill\qed\medskip 

\subsection{A pro-$p$ analogue $\mathrm{DMR}_p^\B$ of the group scheme $\mathsf{DMR}^\B(-)$}\label{section:final}

\begin{defn}\label{defin:GpdmrB} One sets 
$$
\mathrm{DMR}_p^\B:=\mathsf{DMR}^\B(\mathbb Q_p)\cap(\mathbb Z_p^\times\times F_2^{(p)}). 
$$
\end{defn}

\begin{lem}
$\mathrm{DMR}_p^\B$ is a subgroup of $\mathsf G^\B(\mathbb Q_p)
=(\mathbb Q_p^\times\times F_2(\mathbb Q_p),\circledast)$. 
\end{lem}

\proof This follows from the fact that both $\mathsf{DMR}^\B(\mathbb Q_p)$ and $\mathbb Z_p^\times\times F_2^{(p)}$ 
are subgroups of $(\mathbb Q_p^\times\times F_2(\mathbb Q_p),\circledast)$. \hfill\qed\medskip 

\begin{prop}
The natural inclusions yield the following commutative diagram of groups, in which both squares are Cartesian 
$$
\xymatrix{
\mathrm{GT}_p^{\mathrm{op}}\ar@{^{(}->}[r]\ar@{^{(}->}[d] & \mathrm{DMR}_p^\B\ar@{^{(}->}[r]\ar@{^{(}->}[d] &
(\mathbb Z_p^\times\times F_2^{(p)},\circledast)\ar@{^{(}->}[d] \\ 
\mathsf{GT}(\mathbb Q_p)^{\mathrm{op}}\ar@{^{(}->}[r]& \mathsf{DMR}^\B(\mathbb Q_p)\ar@{^{(}->}[r]& 
(\mathbb Q_p^\times\times F_2(\mathbb Q_p),\circledast)}
$$
\end{prop}

\proof The fact that the right square is Cartesian follows from the definition of $\mathrm{DMR}_p^\B$. Then 
\begin{align*}
& \mathsf{GT}(\mathbb Q_p)^{\mathrm{op}}\cap \mathrm{DMR}_p^\B= 
\mathsf{GT}(\mathbb Q_p)^{\mathrm{op}}\cap (\mathsf{DMR}^\B(\mathbb Q_p)\cap (\mathbb Z_p^\times\times F_2^{(p)})) 
\text{ (by the definition of }\mathrm{DMR}_p^\B)  \\
& =\mathsf{GT}(\mathbb Q_p)^{\mathrm{op}}\cap (\mathbb Z_p^\times\times F_2^{(p)}) 
\text{ (by the inclusion of group schemes }\mathsf{GT}(-)^{\mathrm{op}}\subset \mathsf{DMR}^\B(-)) 
\\
& = \mathrm{GT}_p^{\mathrm{op}}\text{ (by Proposition \ref{prop:coup:de:fil}),}
\end{align*}
so that the left square is Cartesian. 
\hfill\qed\medskip 

\setcounter{section}{-1}

\end{document}